\documentclass[12pt,reqno]{amsart}
\usepackage{amsmath, amsfonts, amssymb, amsthm,mathrsfs}
\usepackage{verbatim}
\usepackage{enumerate}
\def\M{\mathfrak M}
\def\m{\mathfrak m}
\def\Z{\mathbb Z}
\def\C{\mathbb C}
\def\N{\mathbb N}

\def\R{\mathbb R}
\def\T{\mathbb T}
\def\E{\mathbb E}

\def\cP{\mathcal P}
\def\cB{\mathcal B}

\def\cD{\mathcal D}

\def\cA{\mathcal A}

\def\cR{{\mathcal R}}

\def\mes{{\mathrm {mes}\;}}

\def\cA{{\mathcal A}}
\def\cB{{\mathcal B}}

\def\m{\mathfrak m}
\def\M{\mathfrak M}
\def\1{{\bf 1}}

\def\pmod #1{\ ({\rm{mod}}\ #1)}

\def\supp{{\mathrm {supp}\;}}

\theoremstyle{plain}
\newtheorem{theorem}{Theorem}

\newtheorem{proposition}[theorem]{Proposition}
\newtheorem{lemma}{Lemma}

\newtheorem{corollary}[theorem]{Corollary}
\theoremstyle{definition}
\newtheorem{Def}{Definition}
\newtheorem*{acknowledgment}{Acknowledgments}

\theoremstyle{remark}
\newtheorem{Rem}{Remark}

\begin{document}
\title
{Vinogradov's theorem with Piatetski-Shapiro primes}

\author{Yu-Chen Sun}
\address {Department of Mathematics and Statistics, University of Turku, 20014 Turku, Finland}
\email{yuchensun93@163.com}
\author{Shan-Shan Du}
\address{The Fundamental Division, Jingling Institute of Technology, Nanjing 211169,
People's Republic of China}
\email{ssdu@jit.edu.cn}
\author{Hao Pan}
\address{School of Applied Mathematics, Nanjing University of Finance and Economics, Nanjing 210046, People's Republic of China}
\email{haopan79@zoho.com}

\keywords{Piatetski-Shapiro prime, ternary Goldbach problem, transference principle, Harman's sieve method}
\subjclass[2010]{Primary 11P32; Secondary 11L20}
\begin{abstract}We prove that,
for any $c_1,c_2,c_3\in(1,41/35)$, every sufficiently large odd number $N$ can be represented as the sum of three primes
$N = p_1 + p_2 +p_3$ such that $p_i = \lfloor n_{i}^{c_i}\rfloor$ for some $n_i \in \N$  for each $1 \leq i \leq 3$. Our arguments are based on a variant of Green's transference principle due to Matom\"aki, Maynard and Shao. We prove a necessary restriction estimate using Bourgain's strategy and employ Harman's sieve method to optimize our upper bound for $c_i$.
\end{abstract}
\maketitle

\section{Introduction}
\setcounter{lemma}{0} \setcounter{theorem}{0}
\setcounter{equation}{0}
The weak Goldbach conjecture asserts that every odd integer greater than $5$ can be represented as the sum of three primes. The well-known Vinogradov's theorem says that the weak Goldbach conjecture is true for every sufficiently large odd integer. In  2013, Helfgott \cite{He} announced a complete proof of Goldbach's weak conjecture.

Nowadays, Vinogradov's three primes theorem has been extended to restrict some summands to special subsequences of primes such as Chen primes, almost equal primes and primes with the form two squares plus one (for the latest result on these, see \cite{MS17,MMS17,Ter18}). It is natural to consider Vinogradov's theorem with Piatetski-Shapiro primes as they from a natural sparse subset of the primes. 

For each non-integral $c>1$, let $\cP$ be the set of primes and
$$
\N^c=\{\lfloor n^c\rfloor:\,n\in\N\},
$$
where $\lfloor x\rfloor$ denotes the integer part of $x$.
 A prime $p$ is called {\it Piatetski-Shapiro} prime, if $p \in \N^c \cap \cP$.

 In 1953, Piatetski-Shapiro \cite{PS53} proved an asymptotic formula for the number of Piatetski-Shapiro primes for $1<c<1.1$. In fact, he proved that the number of Piatetski-Shapiro primes up to $x^c$ and the number of primes up to $x$ are comparable. i.e., for $1<c<1.1$,
\begin{equation}\label{NumPSprime}
\pi_c(x) \sim \frac{x}{c \log x},   
\end{equation}
where $\pi_c(x) = \#\{p\in\N^c\cap \cP: p \leq x^c \}$.
Subsequently, the upper bound for $c$ in this asymptotic formula was improved several times, The current best result is that, (\ref{NumPSprime}) holds for $1< c < 6121/5302 \approx 1.154$ due to Rivat \cite{Ri92}. Later series of papers \cite{Ji93,Ji94, BHR95, K99, RW01} used sieve methods to give a lower bound for the number of Piatetski-Shapiro primes with larger $c$. The current best result is that for $1 < c <243/205 \approx 1.185$, $\pi_c(x) \gg x/\log x$, due to Rivat and Wu \cite{RW01}.
 
On the other hand, in 1992, Balog and Friedlander \cite{BF92} proved Vinogradov's theorem with Piatetski-Shapiro primes. Namely, they obtained that, for $1 <c < 1.05 $, any sufficiently large odd number $N$ is the sum of three primes from $\N^c$. The result of Balog and Friedlander has been improved in \cite{Ri92, Ji95,K97}. The best known asymptotic result is given by Kumchev \cite{K97},  which allows $1 < c < 1.06$ in Vinogradov's three primes theorem with Piatetski-Shapiro primes. However, by employing sieve methods, Jia \cite{Ji95} gave a lower bound result for $1 < c < 16/15 \approx 1.067$.  In fact, all these proofs utilized the following {\it weak Balog-Friedlander condition} 
\begin{Def}[weak Balog-Friedlander condition]\label{BFcondition}
A real number $c \in (1,2)$ is said to satisfy the   {\it weak Balog-Friedlander condition} if for every sufficiently large $N$ and any $A \geq 2$,
\begin{equation}\label{weakbfcondition}
\sum_{\substack{p\leq N\\ p\in\cP\cap\N^c}} cp^{1-\frac1c}\log p\cdot e(p\theta)=
\sum_{\substack{p\leq N\\ p\in\cP}}\log p\cdot e(p\theta)+O_A\bigg(\frac{N}{\log^A N}\bigg).
\end{equation}
\end{Def}
It is called the {\it weak Balog-Friedlander condition} because Balog and Friedlander even obtained a power saving in error term for small $c$. If error term in  (\ref{weakbfcondition}) is $O(N^{1-\delta})$ for some $\delta > 0$, then we call (\ref{weakbfcondition}) the {\it strong Balog-Friedlander condition}.

Note that upper bounds for $c$ in Kumchev's and Jia's works concerning Vinogradov's theorem are much smaller than Rivat and Wu's upper bounds for $c$ in counting the number of Piatetski-Shapiro primes. In this paper we will improve Kumchev's and Jia's results. 

\begin{theorem}\label{main}
For any $c_1,c_2,c_3\in(1,41/35)$, every sufficiently large odd $N$ can be represented as
$$
N=p_1+p_2+p_3,
$$
where $p_1,p_2,p_3$ are primes and $p_i\in\N^{c_i}$ for each $1\leq i\leq 3$. 
\end{theorem}
Here $41/35 \approx 1.171$ should be compared with Jia's value ${16}/{15} \approx 1.067$ and Rivat-Wu's value $243/205 \approx 1.185$, and one should notice that the absolute barrier occurs when $c_i>3$ for all $i \in \{1,2,3\}$.

This result strongly relies on the transference principle, which gives a lower bound for the number of representations. It is natural to ask whether one can obtain an asymptotic formula for large $c$. In fact, for $c \in(1,{73}/{64})$ we can obtain the following asymptotic result.

\begin{theorem}\label{AsyThm}
For any $c_1,c_2,c_3\in(1,{73}/{64})$ and every sufficiently large odd integer $N$,
$$
\sum_{\substack{p_1+p_2+p_3 = N \\ p_i \in \cP \cap \N^{c_i}}}1 = (1+o(1))\frac{\gamma_1\gamma_2\gamma_3\Gamma(\gamma_1)\Gamma(\gamma_2)\Gamma(\gamma_3)}{\Gamma(\gamma_1 + \gamma_2 +\gamma_3)}\frac{\mathfrak{S}(N)N^{\gamma_1 + \gamma_2 +\gamma_3-1}}{\log^3 N},
$$
where $\gamma_i = {1}/{c_i}$, $\Gamma$ denotes the Gamma function and the singular series
\begin{equation}\label{Singular}
\mathfrak{S}(N) = \prod_{p \mid N}\left(1 - \frac{1}{(p-1)^2}\right)\prod_{p \nmid N}\left(1 - \frac{1}{(p+1)^3}\right).    
\end{equation}

\end{theorem}

Here ${73}/{64} \approx 1.141$ should be compared with Kumchev's value $1.06$.

Theorems \ref{main} and \ref{AsyThm} are both based on our new mean-value theorem/restriction estimate (see Theorem \ref{primecWnbT} below). The idea of proving Theorem \ref{primecWnbT} is from Bourgain. In \cite{B89}, Bourgain introduced a strategy for estimating $L_q$ norm, with non-integer $q>0$, for a function which has a ``pseudorandom'' majorant. Our Theorem \ref{primecWnbT} is useful since $L_q$ norm estimates can replace $L_2$ norm which was used in previous papers, e.g. \cite{BF92}. Combining the new $L_q$ norm estimates and the circle method, Theorem \ref{AsyThm} follows. To prove Theorem \ref{main}, we apply the transference principle and Harman's sieve method.

Green \cite{G05} first introduced a transference principle in proving Roth's theorem in a dense set of primes and this method has been applied to study several additive problems. For example, Matom\"aki and Shao(M-S) \cite{MS17} proved the three primes theorem with Chen primes and  Matom\"aki, Maynard and Shao(M-M-S) \cite{MMS17} improved the previous results about Vinogradov's theorem with almost equal summands. The two papers utilize somewhat different techniques. For $n \leq N$, let $\rho^{-}(n) \leq \1_{\cP}(n) \leq \rho^{+}(n)$ . In \cite{MS17} they studied the Fourier transform $\sum_{n \leq N}\rho^{-}(n)e(n \theta)$, but in \cite{MMS17} they studied the Fourier transform $\sum_{n \leq N}\rho^{+}(n)e(n \theta)$ and then combined it with a result on $\sum_{n \leq N} \rho^{-}(n)$. With the help of Harman's sieve method, we employ M-M-S version of the transference principle to prove Theorem \ref{main}.

\begin{Rem}
 The reason that we use M-M-S version of the transference principle instead of using M-S version is that for large $c$ we cannot understand the Fourier transform of $\rho^{-}$ well. Namely, if we use the M-S method, the possible largest $c$ is less than $8/7 \approx 1.143 < 41/35 \approx 1.171$, as we will explain in Remark \ref{rem4}.

  Our result can still be improved slightly, but without a significant new idea one cannot get $1.18$. The authors believe even ${20}/{17} \approx 1.176$ in \cite{BHR95} would be difficult to obtain, which will be illustrated in Remark \ref{rem3}.
\end{Rem}

Our methods can also be used to obtain a Roth-type theorem. We say that a set $B \subset \N^c \cap \cP$ has positive relative upper density if
$$
\limsup_{N \to \infty} \frac{|B \cap [N]|}{|\N^c \cap \cP \cap [N]|}>0.
$$
Mirek \cite{M15} utilized the transference principle to obtain that for any $c \in (1,{72}/{71})$, $B \subset \N^c \cap \cP$ with positive relative upper density contains nontrivial $3$-term arithmetic progressions. We can improve Mirek's result as the following

\begin{theorem}\label{ApRWrang}
For any $c \in (1, {243}/{205})$, $B \subset \N^c \cap \cP$ with positive relative upper density contains nontrivial $3$-term arithmetic progressions.
\end{theorem}
Here our ${243}/{205} \approx 1.185$, coming from the current best result for counting the number of Piateski-Shapiro primes given by Rivat and Wu \cite{RW01}, should be compared with Mirek's value ${72}/{71} \approx 1.014$. 

\begin{acknowledgment}
The first author is grateful to his supervisor Kaisa Matom\"aki for giving explanations for different variants of the transference principle and Harman's sieve method, many useful discussions, reading the paper carefully and giving a lot of helpful comments. The authors also thank Alexander Mangerel Xuancheng Shao and Aled Walker for their helpful comments. Y.-C. Sun was supported by EDUFI funding and UTUGS funding and was working in the Academy of Finland project No. $333707$. S.S. Du was supported by National Natural Science Foundation of China (grant No. $11901259$). H. Pan was supported by National Natural Science Foundation of China (grant No. $12071208$).
\end{acknowledgment}

\section{Notation}\label{notation}
\setcounter{lemma}{0} \setcounter{theorem}{0}
\setcounter{equation}{0}
For convenience, we introduce here some notations which will be used in later sections. 
\begin{itemize}
    \item $\T= \R / \Z =[0,1)$;
    \item $\{t\}$ denotes the fractional part of $t$;
    \item $\psi(t):=\{t\}-\frac{1}{2}$;
    \item $\lfloor t \rfloor$ denotes the integer part of $t$;
    \item $\|t\|:=\min_{k \in \Z}|t-k|$;
    \item $\rho(n) = \1_{\cP}(n)$;
    \item $\rho^{-}$ denotes a minorant of $\rho$, and $\rho^{+}$ denotes a majorant of $\rho$. Namely for any $n \leq N$, we have $\rho^{-}(n) \leq \rho(n) \leq \rho^{+}(n)$;
    \item $\pi(x) = \sum_{p \leq x} \rho(n)$;
    \item $\rho(n,z) = \1_{p \mid n \Rightarrow p \geq z}$;
    \item $e(x):=e^{2\pi i x}$;
    \item $\phi(n)$ denotes the Euler totient function;
    \item $d(n)$ denotes the divisor function;
    \item $n \sim N$ means that $n \in (N, 2N]$;
    \item $\Z_{N} = \Z/N\Z$.
\end{itemize}

For a finitely supported $f: \Z \to \R$, we define its Fourier transform
$$
\widehat{f}(\theta) = \sum_{n \in \Z}f(n)e(n \theta),
$$
and if $f,g$ are finitely supported, we define their convolution as
$$
f*g(n) = \sum_{m}f(m)g(n-m).
$$
We later will use $[N]$ to denote the set $\{1,2,\dots,N\}$. 

Similarly, for any integrable function $f: \T \to \C$, we define its Fourier transform 
$$
\widetilde{f} (n)= \int_{ \T}f(\theta)e(n \theta) d \theta,
$$
and if $f,g:\T \to \C$, we define their convolution as
$$
f*g(\theta) = \int_{\T}f(t)g(\theta-t)dt
$$
and their inner product
$$
\langle f,g \rangle=\int_{\T}f(\theta)\overline{g(\theta)} d \theta.
$$
For an integrable function $f:\T \to \C$, we define its $L_q$ norm (or $q$-norm) for $q>0$ as
$$
\|f\|_q = \left(\int_{\T} |f(\theta)|^q  d \theta \right)^{\frac{1}{q}}.
$$
and its $L_{\infty}$ norm (or $\infty$-norm) as
$$
\|f\|_{\infty} = \sup_{\theta \in \T}|f(\theta)|.
$$
Let $f: \R \to \C$ and $g: \R \to \R_{+}$. We write $f = O(g)$ or $ f \ll g $ if there exists constant $C>0$ such that $|f(x)| \leq C g(x)$ for all $x$ in the domain of $f$. We write $f=o(g)$ if 
$$
\lim_{x \to \infty} \frac{f(x)}{g(x)}=0.
$$

Let $f,g: \R \to \R_{+}$. We say $ g \gg f$ if there exists a constant $C>0$ such that $f(x) \leq C g(x)$ for all $x$ in the domain of $f$. We write $f \asymp g$ if $f \ll g \ll f$.

If the implied constant $C$ depends on some constant $A$, then we use the notation $O_{A}, \ll_{A}, \gg_{A}$.  
\section{Outline}
\setcounter{lemma}{0} \setcounter{theorem}{0}
\setcounter{equation}{0} 
In this section, we will state M-M-S version of the transference principle (Proposition \ref{Tran}) and two theorems corresponding to conditions of the transference principle. Theorem \ref{primecWnbT} is corresponding to the restriction estimate in Proposition \ref{Tran}, and Theorem \ref{LUthm} is corresponding to the mean and pseudorandomness conditions in Proposition \ref{Tran}. By directly applying Theorem \ref{primecWnbT} with the circle method, we will immediately get Theorem \ref{AsyThm}.

Recall definitions of $\rho$ and $\rho^{+}$ in Section \ref{notation}.
\begin{theorem}\label{primecWnbT}
Let $1<c<2$ and $$u>2+\frac{4(c-1)}{2-c}$$ be fixed. Suppose $X$ is a sufficiently large positive number, $1 \leq W \ll \log X$, $N = \lfloor {X}/{W}\rfloor$. Let $A>0$ be a sufficiently large integer, $P=\log^{10A} N$ and $Q = {N}/{\log^{50A} N}$. We define the following Hardy-Littlewood decomposition:
\begin{align*}
&\M_{a,q} = \left\{\theta \in \T: \left|\theta - \frac{a}{q}\right| \leq \frac{1}{qQ}\right\},  \\
&\M = \bigcup_{1 \leq q \leq P}\bigcup_{\substack{a=0 \\ (a,q)=1}}^{q-1}\M_{a,q}, \quad \text{and} \quad \m = \T \setminus \M.
\end{align*}

Suppose that $(b,W)=1$ and $f_{W,X,b}, \nu_{W,X,b}: [N] \to \R$ defined through
\begin{equation}\label{DeffWmb}
f_{W,X,b}(n)=
    \frac{\phi(W)\log X}{W}(Wn + b)^{1- \frac{1}{c}}\rho(Wn+b) \1_{Wn+b \in \N^c}
\end{equation}
and
\begin{equation}\label{UppSieveWmb}
\nu_{W,X,b}(n)=
    \frac{\phi(W)\log X}{W}(Wn + b)^{1- \frac{1}{c}}\rho^{+}(Wn+b)\1_{Wn+b \in \N^c}
\end{equation}
satisfy the following conditions:
\begin{enumerate}[(i)]
   \item For any $0<\epsilon<1$, any $\theta \in \M_{a,q}$ with $(a,q)=1$ and $1 \leq q \leq P$, we have 
   $$
   \widehat{\nu}_{W,X,b}(\theta) \ll \frac{q^{-\epsilon}N}{1 + N \|\theta-a/q\|} + O\left(\frac{N}{\log^{A} N} \right).
   $$
    \item For any $\theta \in \m$, we have $\widehat{\nu}_{W,X,b}(\theta) = O({N}/{\log^{A} N})$.
\end{enumerate}

Then we have
\begin{equation}\label{fWXbunorm}
\int_{\T}|\widehat{f}_{W,X,b}(\theta)|^ud\theta\ll_u N^{u-1}.
\end{equation}
\end{theorem}
We will prove Theorem \ref{primecWnbT} in  the end of Section \ref{ResEs}.
\begin{Rem}\label{ReplcLog}
Theorem \ref{primecWnbT} also holds if we replace the $\log X$ factor in (\ref{DeffWmb}) and (\ref{UppSieveWmb}) by $\log n$. The proof is same as the proof of Theorem \ref{primecWnbT} and we leave this proof to interested reader.
\end{Rem}
Theorem \ref{primecWnbT} has implications for Vinogradov's theorem for Piatetski-Shapiro primes.
\begin{corollary}\label{BFthrPrimes}
Suppose that $c_1,c_2,c_3\in(1,6/5)$ satisfy the weak Balog-Friedlander condition. Then every sufficiently large odd $N$ can be represented as
$$
N=p_1+p_2+p_3,
$$
where $p_1,p_2,p_3$ are primes and $p_i\in\N^{c_i}$ for each $1\leq i\leq 3$. Furthermore, for singular series $\mathfrak{S}$ defined in (\ref{Singular}), we have 
$$
\sum_{\substack{p_1+p_2+p_3 = N \\ p_i \in \cP \cap \N^{c_i}}}1 = (1+o(1))\frac{\gamma_1\gamma_2\gamma_3\Gamma(\gamma_1)\Gamma(\gamma_2)\Gamma(\gamma_3)}{\Gamma(\gamma_1 + \gamma_2 +\gamma_3)}\frac{\mathfrak{S}(N)N^{\gamma_1 + \gamma_2 +\gamma_3-1}}{\log^3 N}.
$$
\end{corollary}
For convenience, here we prove an weighted asymptotic formula, see e.g. Balog-Friedlander \cite[(1.7)]{BF92} and the weighted version implies the unweighted version, which was also stated in \cite{BF92}. Namely, we will prove that for any $A_1>0$,
$$
\sum_{\substack{p_1+p_2+p_3 = N \\ p_i \in \cP \cap \N_{c_i}}}\prod_{i=1}^{3}(c_i p_{i}^{1-\frac{1}{c_i}}\log p_{i}) = \frac{1}{2}\mathfrak{S}(N) N^2 + O_{A_1}\left(\frac{N^2}{\log^{A_1} N}\right).
$$
\begin{proof}
Let $N$ be a large odd integer. We choose $W=1$, $b=0$, $\rho^{+} = \rho$ in Theorem \ref{primecWnbT} and replace $\log N$ by $\log n$ as mentioned in Remark \ref{ReplcLog}, so, for any $c\in \{c_1,c_2,c_3\}$, we define $g_c,h: [N] \to \R$ through
$$g_c(n) = cf_{1,X,0}(n) = c f_{1,N,0}(n)  = cn^{1-\frac1c}\log n \1_{n\in\cP\cap\N^c} \quad\text{and} \quad  h(n) =\log n \1_{n \in \cP}. $$
By the weak Balog-Friedlander  condition, for any sufficiently large $A_0> A_1/{(3-u)}$, we have
\begin{equation}\label{BFagin}
\widehat{g}_c(\theta) = \widehat{h}(\theta) + O_{c,A}\left(\frac{N}{\log^{A_0} N} \right).
\end{equation}
It is standard to verify that $\widehat{h}(\theta)$ satisfies the two conditions in Theorem \ref{primecWnbT}(or one can refer to Green's arguments \cite[Lemmas 4.4 and 4.9]{G05}). Hence so does $\widehat{g}_c(\theta)$.
Since $c_i \in (1,6/5)$ for each $i \in \{1,2,3\}$, Theorem \ref{primecWnbT} implies that there exists $u \in (2,3)$ such that, for each $i \in \{1,2,3\}$
\begin{equation}\label{hgunorm}
\|\widehat{g}_{c_i}\|_{u}^{u}, \|\widehat{h}\|_{u}^{u} \ll N^{u-1}.    
\end{equation}
 Write $g_i(n) = g_{c_i}(n)$, then by telescoping and Vingroadov's three primes asymptotic formula \cite[Theorem 3.4]{Vau97} we have
\begin{align}
& g_1*g_2*g_3(N) \notag\\
= & h*h*h(N) + h*h*(g_3 - h)(N)+ h*(g_2 - h)*g_3(N) + (g_1 - h)*g_2*g_3(N) \notag\\
= & \frac{1}{2}(1+o(1))\mathfrak{S}(N)N^2 +  \int_{\T}\widehat{h}^2(\theta)(\widehat{g}_3(\theta) - \widehat{h}(\theta)) e(- N \theta) d \theta \notag \\
+ & \int_{\T}\widehat{h}(\theta)(\widehat{g}_2(\theta) - \widehat{h}(\theta))\widehat{g}_3(\theta) e(- N \theta) d \theta + \int_{\T}(\widehat{g}_1(\theta) - \widehat{h}(\theta))\widehat{g}_2(\theta)\widehat{g}_3(\theta) e(- N \theta) d \theta \notag \\
=: & \frac{1}{2}\mathfrak{S}(N)N^2 + O_{A_0}\left(\frac{N^2}{\log^{A_0} N}\right) + S_1 +S_2 +S_3, \label{covg}
\end{align}
say. Now we just need to prove that $S_1, S_2$ and $S_3$ in (\ref{covg}) are $O(N^2/\log^{A_1} N)$. By (\ref{BFagin}), H\"older's inequality and (\ref{hgunorm}), we have 
\begin{align*}
S_1 & \ll \sup_{\theta \in \T}|\widehat{g_3}(\theta) - \widehat{h}(\theta)|^{3-u} \int_{\T}|\widehat{h}(\theta)|^2 |\widehat{g_3}(\theta) - \widehat{h}(\theta)|^{u-2} d\theta \\
& \ll \frac{N^{3-u}}{\log^{(3-u)A_0} N} \|\widehat{h}\|_{u}^2\|\widehat{g_3} - \widehat{h}\|_{u}^{u-2} \ll \frac{N^{3-u}}{\log^{(3-u)A_0} N} N^{u-1}=O_{A_1}\left(\frac{N^2}{\log^{A_1} N}\right).
\end{align*}
Similarly, one can prove that $S_2$ and $S_3$ are also $O_{A_1}(N^2/\log^{A_1} N)$, so our claim follows.
\end{proof}
\begin{proof}[Proof of Theorem \ref{AsyThm}]
Kumchev \cite{K97} showed that the {\it weak Balog-Friedlander condition} is valid for
 every $c \in (1,73/64)$. Hence, we obtain Theorem \ref{AsyThm} by applying Corollary \ref{BFthrPrimes}.
\end{proof}
For proving our Theorem \ref{main}, we need to utilize the following transference principle which has been proved by Matom\"aki, Maynard and Shao in \cite[Proposition 3.1]{MMS17}.
\begin{proposition}[Transference principle]\label{Tran}
Let $\epsilon, \eta \in(0,1)$. Let $N$ be a positive integer and Let $f_1,f_2,f_3:[N]\to \R_{\geq 0}$ be functions, with each $f \in \{f_1, f_2, f_3\}$ satisfying the following assumptions:
\begin{enumerate}[(i)]
    \item({\bf mean condition}) For each arithmetic progression $P \subset [N]$ with $|P| \geq \eta N$ we have $\E_{n \in P}f(n) \geq 1/3 + \epsilon$;
    \item({\bf pseudorandomness}) There exists a majorant $\nu:[N] \to \R_{\geq 0}$ with $f \leq \nu$ pointwise, such that $\|\widehat{\nu}-\widehat{\1}_{[N]}\|_{\infty} \leq \eta N$;
    \item({\bf restriction estimate})We have $\|\widehat{f}\|_{q} \leq K N^{1 - 1/q}$ for some fixed $q, K$ with $K \geq 1$ and $2 < q <3$.
\end{enumerate}
Then for each $n \in [N/2, N]$ we have
$$
f_1 * f_2 * f_3(n) \geq (c(\epsilon) - O_{\epsilon, K, q}(\eta))N^2,
$$
where $c(\epsilon) >0$ is a constant depending only on $\epsilon$.
\end{proposition}
We will choose $f_i = f_{W,X,b_i}$ and $\nu_{i} = \nu_{W,X,b_i}$ with $f_{W,X,b_i}$ and $\nu_{W,X,b_i}$ as in (\ref{DeffWmb}) and (\ref{UppSieveWmb}) with an appropriate $\rho^{+}$ we shall construct. Our main task is then to show that these choices satisfy the following theorem.
\begin{theorem}\label{LUthm}
Let $X$ be a large positive integer, $W=\prod_{p \leq w}p$ where $w=0.1\log\log X$, $\rho(n)={\bf 1}_{\cP}(n)$. Then for any $c \in (1,41/35)$ there exist a majorant $\rho^{+}(n) \geq \rho(n)$ and constants $\alpha^{+} > \alpha^{-} >0$ such that the following four properties hold.
\begin{enumerate}[(i)]
    \item\label{UppForRestrction}  The two conditions in Theorem \ref{primecWnbT} hold for any $1 \leq b \leq W$ with $(b,W)=1$.
    \item \label{PsedoRandom} For every residue class $l \pmod{W}$ with $(l,W)=1$, and every $\theta \in \T$, we have
    $$
    \left| \sum_{\substack{n \in [X]\cap \N^c\\ n \equiv l \pmod{W}}} c n^{1-\frac{1}{c}}\rho^{+}(n)e(n\theta) -  \frac{\alpha^{+}}{\log X}\cdot \frac{W}{\phi(W)}\sum_{\substack{n \in [X] \\ n \equiv l \pmod{W}}}e(n \theta)\right| = o\left(\frac{X}{\phi(W)\log X}\right).
    $$
    \item \label{LowSieveInAps} Let $0 < \eta < 1$ and $X$ be sufficiently large in terms of $\eta$ and $c$. For any residue class $l \pmod{d}$ with $(l,d)=1$, $d \leq \log X$, and  arithmetic progression $P \subset \{n\in [X] : n \equiv l \pmod{d}\}$ with length at least $\eta {X}/{d}$, we have 
    $$
    \sum_{n \in P\cap \N^c } c n^{1- \frac{1}{c}} \rho(n) \geq  \alpha^{-}\frac{d}{\phi(d)}\frac{|P|}{\log X}.
    $$
    \item \label{UminusL} $\alpha^{+} - 3 \alpha^{-} > 0$.
\end{enumerate}
\end{theorem}

\begin{proof}[Proof of Theorem \ref{main}]
Let $n_0$ be a large odd integer. Choose $X$ such that $n_0 \in (\frac{2}{3}X, \frac{5}{6}X]$ and let $\rho$, $\rho^{+}$, $\alpha^{-}$, $\alpha^{+}$, $W$, $w$ be as in the statement of Theorem \ref{LUthm}. 
Choose $b_1,b_2,b_3 \in [W]$ with $(b_i,W)=1$ such that $b_1 + b_2 + b_3 \equiv n_0 \pmod{W}$.
Let $N = \left\lfloor{X}/{W}\right\rfloor$ and let $f_i, \nu_i: [N] \to \R$ be defined through
\begin{equation}\label{Deff}
f_i(n)=
\begin{cases}
    c_i\cdot\frac{\log X}{\alpha^{+}} \cdot \frac{\phi(W)}{W}(Wn + b_i)^{1- \frac{1}{c_i}}\rho(Wn+b_i), & \text{ if } Wn+b_i \in \N^{c_i},  \\
    0, & \text{ otherwise},
\end{cases}
\end{equation}
and
\begin{equation}\label{Defnu}
\nu_i(n)=
\begin{cases}
c_i\cdot\frac{\log X}{\alpha^{+}} \cdot \frac{\phi(W)}{W}(Wn + b_i)^{1- \frac{1}{c_i}}\rho^{+}(Wn+b_i), & \text{ if } Wn+b_i \in \N^{c_i},  \\
    0, & \text{ otherwise}.
\end{cases}
\end{equation}
Then $f_i \leq \nu_i$ since $\rho \leq \rho^{+}$. We shall apply Proposition \ref{Tran} with these choices, so we need to verify its assumptions. To prove the Fourier uniformity of $\nu_i$ (Proposition \ref{Tran} (ii)), observe that
\begin{align*}
& \sum_{n \in [N]}\nu_i(n)e(n \theta)  = c_i\cdot\frac{\log X}{\alpha^{+}} \cdot \frac{\phi(W)}{W}\sum_{\substack{n \leq N \\ Wn+b_i \in \N^{c_i}}}(Wn+b_i)^{1-\frac{1}{c_i}}\rho^{+}(Wn+b_i)e(n \theta) \\
= & c_i\cdot\frac{\log X}{\alpha^{+}} \cdot e\left(-\frac{b_i}{W} \theta\right)\frac{\phi(W)}{W} \sum_{\substack{n \in [X]\cap\N^{c_i} \\ n \equiv b_i \pmod{W}}}n^{1-\frac{1}{c_i}}\rho^{+}(n)e\left(\frac{n\theta}W\right) + o(N), 
\end{align*}
and similarly 
$$
\sum_{n \in [N]}e(n\theta) = e\left(-\frac{b_i}{W}\theta\right)\sum_{\substack{n \in [X] \\ n \equiv b_i \pmod{W}}}e\left(\frac{n\theta}{W}\right) + O(1)
$$
for any $\theta \in \T$. Comparing the two equations above and using Theorem \ref{LUthm} (\ref{PsedoRandom}), we obtain 
$$
\left| \sum_{n \in [N]}\nu_i(n)e(n \theta)  - \sum_{n \in [N]}e(n\theta) \right| = o(N).
$$
Hence Proposition \ref{Tran} (ii) holds. Let us now establish Proposition \ref{Tran} (i). For any arithmetic progression $P \subset [N]$ of length $\geq \eta N$, by Theorem \ref{LUthm} (\ref{LowSieveInAps}), we have the lower bound 
$$
\sum_{n \in P}f_i(n) = c_i \frac{\log X}{\alpha^{+}}\frac{\phi(W)}{W}\sum_{n \in Q \cap \N^{c_i}}n^{1-\frac{1}{c_i}}\rho(n) \geq \frac{\alpha^{-}}{\alpha^{+}}|Q| = \frac{\alpha^{-}}{\alpha^{+}}|P| > \frac{1}{3}|P|,
$$
by (\ref{UminusL}), where $Q = \{Wn+b: n \in P\}$ (Note that if $We$ is the common difference of $Q$, then, by $|P| \geq \eta N$, $e \leq 1/\eta$ and hence $We/\phi(We)=W/\phi(W)$).

Finally, by Theorem \ref{LUthm} (\ref{UppForRestrction}), Theorem \ref{primecWnbT} is applicable and so Proposition \ref{Tran} (iii) holds. We can apply Proposition \ref{Tran} to find that for each 
$n \in [N/2, N]$, there exist $n_1,n_2,n_3$ with each $n_i$ in the support of $f_i$ such that
$$
n = n_1 +n_2 +n_3.
$$ 
In particular, each $Wn_i + b_i$ is a prime in $\N^{c_i}$, and we have the representation
$$
Wn + b_1 + b_2 + b_3 =  Wn_1 + b_1 +Wn_2+ b_2 +Wn_3+ b_3=  p_1 + p_2 + p_3.
$$
Setting $n = (n_0 - b_1 -b_2 -b_3)/{W} $, our claim follows.
\end{proof}

\begin{Rem}
Using the majorant constructed in Theorem \ref{LUthm} combining with \cite[Proposition 5.1]{GT06}, one can immediately obtain that for $1<c<41/35$, $\N^c \cap \cP$ contains non-trivial three term arithmetic progressions. We will give a brief proof for Theorem \ref{ApRWrang} in Section \ref{Roth}.
\end{Rem}

The rest of the paper is organized as follows. In Section \ref{ResEs}, we will prove Theorem \ref{primecWnbT} by van der Corput methods and Bourgain's strategy. In Sections \ref{LBsieve} and \ref{UBsieve}, we will use known type I and II estimates and Harman's sieve method to construct lower and upper bound sieves $\rho^{-}(n)$,  $\rho^{+}(n)$ whose corresponding coefficients $\alpha^{-}, \alpha^{+}$ satisfy $3 \alpha^{-} - \alpha^{+} > 0$. In Section \ref{UBFourier}, we will study the Fourier transform of the upper bound sieve to prove Theorem \ref{LUthm} (i) and (ii). In Section \ref{Roth}, we will prove Theorem \ref{ApRWrang}.

\section{restriction estimates}\label{ResEs}
\setcounter{lemma}{0} \setcounter{theorem}{0}
\setcounter{equation}{0}
The following lemma is the well-known van der Corput inequality, which is very effective when dealing with exponential sum estimates involving $n^c$ with non-integer $c$.

\begin{lemma}\label{vanderCorput} Let $X \geq Y >0$. Suppose that $\Delta>0$ and
$$
|f''(x)|\asymp\Delta
$$
for any $x\in[X,X+Y]$. Then
$$
\sum_{X < n\leq X+Y}e\big(f(n)\big)\ll Y\Delta^{\frac12}+\Delta^{-\frac12}.
$$
\end{lemma}
\begin{proof}
See \cite[Theorem 2.2]{GK91}.
\end{proof}
The following lemma is the well-known Fourier expansion of $\psi(t)$, see e.g. \cite[(2.6)]{BF92}.
\begin{lemma}\label{t[t]12}
For each $H\geq 2$ and $t \in \R$,
$$
\psi(t)=-\frac1{2\pi i}\sum_{0<|h|\leq H}\frac{e(ht)}{h}+O\bigg(\min\bigg\{1,\frac{1}{H\|t\|}\bigg\}\bigg).
$$
Furthermore, we have
$$
\min\bigg\{1,\frac{1}{H\|t\|}\bigg\}=\sum_{h=-\infty}^{\infty}b_he(ht),
$$
where
$$
b_h\ll\min\bigg\{\frac{\log H}{H},\frac{H}{h^2}\bigg\}.
$$
\end{lemma}

The next lemma claims that $cn^{1-\frac1c}\1_{n \in \N^{c}}$ is pseudorandom with a good error term when $c$ is close to $1$.
\begin{lemma}\label{PseShap} Let $N$ be a sufficiently large integer. For each $\theta\in\T$ and $c \in (1,2)$,
\begin{equation}
\sum_{\substack{n\leq N\\ n\in\N^c}}cn^{1-\frac1c}\cdot e(n\theta)=
\sum_{\substack{n\leq N}}e(n\theta)+O\left(N^{\frac{3}{2}-\frac1c}\log N\right).
\end{equation}
\end{lemma}
\begin{proof}
By dyadic splitting, we only need to prove that
\begin{equation}\label{cn11centhetansN}
\sum_{\substack{n\sim N\\ n\in\N^c}}cn^{1-\frac1c}\cdot e(n\theta)=\sum_{\substack{n\sim N}}e(n\theta)+O\left(N^{\frac{3}{2}-\frac1c}\log N\right).
\end{equation}
It is easy to verify that $n\in\N^c$ if and only if
$$
\lfloor -n^{\frac1c} \rfloor-\lfloor-(n+1)^{\frac1c}\rfloor=1.
$$
Clearly
$$
\lfloor-n^{\frac1c} \rfloor=-n^{\frac1c}-\psi(-n^{\frac1c})-\frac12.
$$
Hence
\begin{align*}
&\sum_{\substack{n\sim N\\ n\in\N^c}}cn^{1-\frac1c}\cdot e(n\theta)=
\sum_{n\sim N} e(n\theta)\cdot cn^{1-\frac1c}\left(\lfloor-n^{\frac1c}\rfloor-\lfloor-(n+1)^{\frac1c}\rfloor\right)\\
=&\sum_{n\sim N} e(n\theta)\cdot cn^{1-\frac1c}
\left(\psi(-(n+1)^{\frac1c})-\psi(-n^{\frac1c})\right)+
\sum_{n\sim N} e(n\theta)\cdot cn^{1-\frac1c}\left((n+1)^{\frac1c}-n^{\frac1c}\right)\\
=&\sum_{n\sim N} e(n\theta)\cdot cn^{1-\frac1c}
\left(\psi(-(n+1)^{\frac1c})-\psi(-n^{\frac1c})\right)+
\sum_{n\sim N}e(n\theta)+O(1).
\end{align*}
According to Lemma \ref{t[t]12},
\begin{align}
&\sum_{n\sim N} e(n\theta)\cdot cn^{1-\frac1c}
\left(\psi(-(n+1)^{\frac1c})-\psi(-n^{\frac1c})\right) \label{SumExpDoublePsi}\\
=&-\frac{1}{2 \pi i}\sum_{0 < |h|\leq H}\frac{1}{h}\sum_{n\sim N} e(n\theta)\cdot cn^{1-\frac1c}\left(e(-h(n+1)^{\frac1c})-e(-hn^{\frac1c})\right) \notag\\
&+ O\bigg(N^{1-\frac1c}\cdot\sum_{N < n \leq  2N+1}\min\bigg\{1,\frac{1}{H\|n^{\frac1c}\|}\bigg\}\bigg)\notag\\
=&-\frac{1}{2 \pi i}\sum_{0 < |h|\leq H}\frac{1}{h}\sum_{n\sim N} e(n\theta)\cdot cn^{1-\frac1c}\left(e(-h(n+1)^{\frac1c})-e(-hn^{\frac1c})\right) \notag\\
&+ O\bigg(N^{1-\frac1c}\cdot\sum_{n \sim N}\min\bigg\{1,\frac{1}{H\|n^{\frac1c}\|}\bigg\}\bigg) + O(N^{1-\frac1c}).\notag
\end{align}

Clearly
$$
e(-h(n+1)^{\frac1c})-e(-hn^{\frac1c})=-\frac{2\pi ih}{c}\int_{\T}(n+u)^{\frac1c-1}e(-h(n+u)^{\frac1c})d u.
$$
Hence the main term on the right hand side of (\ref{SumExpDoublePsi}) is 
\begin{equation}\label{UpperDoulePsi}
 \ll \sup_{u \in \T}\left|\sum_{0<|h|\leq H}\sum_{n \sim N}n^{1-1/c}\frac{e(n \theta - h(n+u)^{1/c})}{(n+u)^{1-1/c}}\right|.  
\end{equation}

For $u\in \T$, we have
$$
\frac{d^2}{dx^2}(x\theta\pm h(x+u)^{\frac1c})=\mp\frac{c-1}{c^2}\cdot\frac{h}{(x+u)^{2-\frac1c}}.
$$
Hence Lemma \ref{vanderCorput} implies
\begin{align*}
\sum_{n \sim N} e\left(n\theta-h(n+u)^\frac1c \right)\ll N\cdot\frac{h^{\frac12}}{N^{1-\frac1{2c}}}+\frac{N^{1-\frac1{2c}}}{h^{\frac12}}.
\end{align*}
It follows that
\begin{align*}
\sum_{n\sim N} \frac{n^{1-\frac1c}}{(n+u)^{1-\frac1c}}\cdot e\left(n\theta-h(n+u)^\frac1c \right)\ll N^{\frac1{2c}}h^{\frac12}+N^{1-\frac1{2c}}h^{-\frac12}+1.
\end{align*}
Hence (\ref{UpperDoulePsi}) is
$$
\ll N^{\frac1{2c}}H^{\frac32}+N^{1-\frac1{2c}}H^{\frac12}+H.
$$

Let us turn to the error term in (\ref{SumExpDoublePsi}). Let $b_h$ be as in Lemma \ref{t[t]12}. By Lemmas \ref{vanderCorput} and \ref{t[t]12}, we obtain
\begin{align}\label{N11c1Hn1c}
&N^{1-\frac1{c}}\sum_{n\sim N}\min\bigg\{1,\frac{1}{H\|n^{\frac1c}\|}\bigg\}\notag\\
=& N^{1-\frac1{c}}\sum_{h=-\infty}^\infty b_h\sum_{n\sim N} e(hn^{\frac1c})\ll
N^{1-\frac1{c}}\sum_{h=-\infty}^\infty |b_h|\cdot \left(h^{\frac12}N^{\frac1{2c}}+N^{1-\frac1{2c}}h^{-\frac12}\right)\notag\\
\ll&
N^{1-\frac1{c}}\sum_{|h|<H} \frac{\log H}{H}\cdot \left(h^{\frac12}N^{\frac1{2c}}+N^{1-\frac1{2c}}h^{-\frac12}\right)+
N^{1-\frac1{c}}\sum_{|h|\geq H} \frac{H}{h^2}\cdot \left(h^{\frac12}N^{\frac1{2c}}+N^{1-\frac1{2c}}h^{-\frac12}\right)\notag\\
\ll&N^{1-\frac1{2c}}H^{\frac12}\log H+N^{2-\frac3{2c}}H^{-\frac12}\log H+N^{1-\frac1{2c}}H^{\frac12}+
N^{2-\frac3{2c}}H^{-\frac12}.
\end{align}
Finally, letting $H=N^{1-\frac1c}$, we obtain that
$$
\sum_{0<|h|\leq H}\sum_{n\sim N} e(n\theta)\cdot cn^{1-\frac1c}\left(e(-h(n+1)^{\frac1c})-e(-hn^{\frac1c})\right)\ll N^{\frac32-\frac1c},
$$
and
$$
N^{1-\frac1{c}}\sum_{n\sim N}\min\bigg\{1,\frac{1}{H\|n^{\frac1c}\|}\bigg\}\ll
N^{\frac32-\frac1c}\log N.
$$
Thus (\ref{cn11centhetansN}) follows from (\ref{SumExpDoublePsi}).
\end{proof}

The following lemmas are motivated by Bourgain's strategy \cite{B89}, which can give $L_q$ norm estimates when $q>0$ is not an integer.

\begin{lemma}\label{Rdelta}
 Let $N$ be a sufficiently large integer and $f:[N] \to \mathbb{C}$ be a function. We define
 $$
\cR_\delta:=\{\theta\in \T:\,|\widehat{f}(\theta)|>\delta N\}.
$$
Let $v>0$. Assume that, for any $\epsilon > 0$ and every $\delta \in (0,1)$ we have
\begin{equation}\label{Rdeltae}
\mes(\cR_\delta)\ll_{\epsilon}\frac{1}{\delta^{v+\epsilon}N}.
\end{equation}
Then for any $u > v$
$$
\int_{\T}|\widehat{f}(\theta)|^u d\theta\ll_u N^{u-1}.
$$
\end{lemma}
\begin{proof}
Let $\epsilon_0 \in (0,u-v)$. Then
\begin{align}\label{hatfuRdelta}
\int_{\T}|\widehat{f}(\theta)|^ud\theta\leq&\sum_{j\geq 1}\bigg(\frac{N}{2^{j-1}}\bigg)^u\cdot\mes\bigg(\bigg\{\theta\in \T:\,\frac{N}{2^{j}}<|\widehat{f}(\theta)|\leq
\frac{N}{2^{j-1}}\bigg\}\bigg)\notag\\
\ll_{\epsilon_0}&\sum_{j\geq 1}\bigg(\frac{N}{2^{j-1}}\bigg)^u\cdot\frac{1}{(\frac1{2^j})^{v +\epsilon_0 }N}=2^uN^{u-1}\sum_{j\geq 1}\frac{1}{2^{(u-v-\epsilon_0)j}}\ll N^{u-1}.
\end{align}
\end{proof}

We first use Bourgain's strategy \cite{B89} to prove the following mean-value result for functions whose majorants are ``pseudorandom''.
\begin{lemma}\label{BourInt}
Let $f, \nu:\,[N]\to\C$ functions with $|f|\leq \nu$ and let $K < N$. Let $u_0, u_1>0$ be such that the following assumptions hold:
\begin{enumerate}[(i)]
    \item $\int_{\T}|\widehat{f}(\theta)|^{u_0} d \theta\ll K N^{u_0-1}$.
    \item For any $\theta \in \T$ and $\epsilon > 0$, we have $$\widehat{\nu}(\theta) \ll \frac{N}{1 + N \|\theta\|} + o(NK^{-\frac{2}{u_1 + \epsilon}} ).$$
    \item $u_0 + u_1 \geq 2$.
\end{enumerate}
Then, for any $u > u_0 + u_1$,
$$
\int_{\T}|\widehat{f}(\theta)|^u d\theta\ll_u N^{u-1}.
$$
\end{lemma}
\begin{proof}
Let
$$\cR_\delta:=\{\theta\in \T:\,|\widehat{f}(\theta)|>\delta N\}$$ for any $\delta\in (0,1)$.
According to Lemma \ref{Rdelta}, we only need to show that
\begin{equation}
\mes(\cR_\delta)\ll_{\epsilon}\frac{1}{\delta^{u_0+u_1+\epsilon}N}
\end{equation}
for any $\epsilon>0$ and $\delta\in (0,1)$.

Suppose first that 
$$
\delta < K^{-\frac{1}{u_1 + \epsilon}}.
$$
Then, by the definition of $\cR_{\delta}$ and assumption (i),
\begin{align*}
(\delta N)^{u_0}\cdot\mes(\cR_\delta)\leq\int_{\T}|\widehat{f}(\theta)|^{u_0}d\theta
\ll K N^{u_0 - 1},
\end{align*}
so
$$
\mes(\cR_\delta)\ll \frac{K }{\delta^{u_0} N}\leq
\frac1{\delta^{u_0+u_1+\epsilon} N}.
$$
From now on we can assume that $\delta \geq K^{-\frac{1}{u_1 + \epsilon}}$. There exist $\theta_1,\ldots,\theta_R\in \cR_\delta$ that are $N^{-1}$-spaced such that $\mes(\cR_{\delta}) \ll {R}/{N}$.
Since $\theta_r\in \cR_\delta$, we have $$|\widehat{f}(\theta_r)| > \delta N$$ for $r=1, \dots, R$. It follows that
$$
R^2\delta^2N^2\leq \bigg(\sum_{r=1}^R|\widehat{f}(\theta_r)|\bigg)^2.
$$
Since $|f|\leq \nu$, we can write, for each $n$, $$f(n)=a_n\nu(n)$$ with $|a_n|\leq 1$. Further, for $1\leq r\leq R$, write
$$
|\widehat{f}(\theta_r)|=b_r\widehat{f}(\theta_r)
$$
where $|b_r|=1$. Then by the Cauchy-Schwarz inequality (our argument corresponds to using duality principle), we have 
\begin{align}\label{Dual}
R^2\delta^2N^2\leq&\bigg(\sum_{r=1}^Rb_r\sum_{n\in\Z} a_n\nu(n)e(\theta_rn)\bigg)^2 \notag \\
\leq&\bigg(\sum_{n\in\Z}|a_n|^2\nu(n)\bigg)\cdot
\bigg(\sum_{n\in\Z}\nu(n)\bigg|\sum_{r=1}^Rb_r  e(\theta_rn)\bigg|^2\bigg) \notag\\
\leq&\bigg(\sum_{n\in\Z}\nu(n)\bigg)\cdot
\sum_{1\leq r,r'\leq R}b_r\overline{b_{r'}}\sum_{n\in\Z}\nu(n)  e\big((\theta_r-\theta_{r'})n\big) \notag \\
\leq &2N
\sum_{1\leq r,r'\leq R}\big|\widehat{\nu}(\theta_r-\theta_{r'})\big|. 
\end{align}
Recall $\delta \geq K^{-\frac{1}{u_1 + \epsilon}}$ and let $\theta_{r,r'} = \theta_r - \theta_{r'}$. Hence by assumption (ii),
\begin{align*}
\delta^2 \ll&\frac{1}{R^2}
\sum_{\substack{1\leq r,r'\leq R
}}\frac{1}{1+N\|\theta_{r,r'}\|} \ll \frac{1}{R} + \frac{1}{R^2}\sum_{\substack{1\leq r,r'\leq R \\ r \neq r'}}\frac{1}{N\|\theta_{r,r'}\|} \\\ll&\frac{1+\log R}{R}\ll R^{-1+ \min\{\epsilon^2, 1/100\}}.
\end{align*}
Hence
$$
R \ll \frac{1}{\delta^{2 + \epsilon}},
$$
which implies that 
$$\mes(\cR_\delta) \ll \frac{1}{\delta^{2 +\epsilon} N} \ll \frac{1}{\delta^{u_0 + u_1 +\epsilon} N}.$$
\end{proof}
Next lemma is an application of Lemmas \ref{PseShap} and \ref{BourInt}.
\begin{lemma}\label{falphaeta}
Let $X,W,N$ be as in Theorem \ref{primecWnbT}. Let $b \in [W]$. For a fixed $c \in (1,2)$ we define 
$\tau_{W,X,b}:[N] \to \R$ through
$$
\tau_{W,X,b}(n):= c(Wn+b)^{1-\frac1c}\1_{Wn+b \in \N^c}.
$$
If $f:\,[N]\to \R$ is a function with $|f|\leq\tau_{W,X,b}$ and $$u>2+\frac{4(c-1)}{2-c},$$
then
\begin{equation}
\int_{\T}|\widehat{f}(\theta)|^ud\theta\ll_u N^{u-1}.
\end{equation}
\end{lemma}
 
\begin{proof}
By Parseval's identity and the definition of $\tau_{W,X,b}$,
$$
\int_{\T}|\widehat{f}(\theta)|^2d\theta\leq \sum_{n}\tau_{W,X,b}(n)^2
\ll (WN)^{2-\frac1c} \ll N \cdot N^{1-\frac{1}{c}}\log^{2-\frac{1}{c}} N.
$$
On the other hand, by orthogonality of additive characters and Lemma \ref{PseShap}, for $\theta\in \T$ we have
\begin{align}
&\widehat{\tau}_{W,X,b}(\theta)= \sum_{\substack{n \leq N \\ Wn+b \in \N^c}}c(Wn+b)^{1-\frac{1}{c}}e(n \theta) \\
&= \sum_{\substack{m \leq X \\ m \equiv b \pmod{W} \\ m \in \N^c} }c m^{1-\frac{1}{c}}e\left(\frac{m-b}{W}\theta\right) + O(X^{1-\frac{1}{c}})\notag\\
&=\frac1W\sum_{k=0}^{W-1}
e\bigg((\theta+k)\cdot\frac{-b}{W}\bigg)\cdot\widehat{\tau}_{1,X,0}\bigg(\frac{\theta+k}{W}\bigg) + O(X^{1-\frac{1}{c}}) \notag\\
&=\frac1W\sum_{k=0}^{W-1}
e\bigg((\theta+k)\cdot\frac{-b}{W}\bigg)\sum_{\substack{m\leq X}}e\bigg(m\cdot\frac{\theta+k}{W}\bigg)+O(X^{\frac{3}{2}-\frac1c}\log X) \notag\\
&=\sum_{\substack{m\leq X\\ m\equiv b\pmod{W}}}e\bigg(\theta\cdot\frac{m-b}{W}\bigg)+O(X^{\frac{3}{2}-\frac1c}\log X) \notag\\
&= \sum_{\substack{n\leq N}}e(\theta\cdot n) + O(X^{\frac{3}{2}-\frac1c }\log X) = \sum_{\substack{n\leq N}}e(\theta\cdot n) + O(N^{\frac{3}{2}-\frac1c }\log^{\frac{5}{2}-\frac{1}{c}} N).\label{EroSharExp}
\end{align}
Thus we can apply Lemma \ref{BourInt}, with $u_0 = 2$, $K = N^{1-\frac{1}{c}}\log^2 N$ and 
$$
u_1 = {4(c-1)}/{(2-c)}.
$$
Note that, for any $\epsilon>0$, 
$$
1 - \frac{2(1-\frac{1}{c})}{u_1 + \epsilon}> 1 - \frac{2(2-c)}{4c} =\frac{3}{2} - \frac1c.
$$
Thus, for any $\epsilon > 0$,
$$
N^{\frac{3}{2}-\frac{1}{c}}\log^{\frac{5}{2}-\frac{1}{c}}N =o(N K^{-\frac{2}{u_1+\epsilon}}),
$$
so the error term in (\ref{EroSharExp}) is $o(NK^{-\frac{2}{u_1 + \epsilon}})$ for $\epsilon>0$. Hence our claim follows by Lemma \ref{BourInt}.
\end{proof}
In the following lemma, we construct a good function that will be used in the proof of Lemma \ref{BourCir}.
\begin{lemma}\label{defsig}
Let $N>R>0$ are sufficiently large integers. Suppose that $\theta_1, \theta_2, \dots, \theta_R$ are $1/N$  spaced. There exists a function $\sigma: \T \to \R$ such that the following four conditions hold
\begin{enumerate}[(i)]
    \item $\sigma(\theta)\in[0,1000]$ for each $\theta\in \T$;
    \item $\supp\widetilde{\sigma}\subseteq[-N,N]$;
    \item $\|\sigma\|_{1} \ll R/N$;
    \item $\sigma(\theta)\geq 1$ if $\|\theta-\theta_r\|\leq 1/N$ for some $1\leq r\leq R$.
\end{enumerate}
\end{lemma}

\begin{proof}
Let $M = \lfloor N/2 \rfloor$. In fact, we can define $\sigma(\theta)$ from ``Fej\'er'' perspective as the following
\begin{align}
\sigma(\theta) & :=\frac{10}{M} \sum_{1 \leq r \leq R} \sum_{|m|\leq M}\left(1- \frac{|m|}{M} \right)e\left(m\left(\theta - \theta_r\right)\right) \label{DefSigma} \\
& = \frac{10}{M}\sum_{1 \leq r \leq R}\frac{\sin^2 (\pi M (\theta - \theta_r))}{M \sin^2 (\pi (\theta - \theta_r))} \label{FejerSin}\\
& \leq \frac{100}{M} \sum_{1 \leq r \leq R} \min\left\{M, \frac{1}{M\|\theta - \theta_r\|^2}\right\} \label{FejerUpp}
\end{align}
By (\ref{FejerSin}) and (\ref{FejerUpp}), we have
$$
0 \leq \sigma(\theta) \leq \frac{100}{M} M\left(1 + \sum_{1 \leq r \leq R}\frac{2}{r^2}\right) \leq 1000,
$$
so (i) follows. (ii) quickly follows from the definition of $\widetilde{\sigma}$ and (\ref{DefSigma}). Since $\sigma \geq 0$, we have 
$$
\|\sigma\|_1 = \int_{\T}\sigma(\theta) d \theta = \frac{10R}{M} \ll \frac{R}{N},
$$
which implies (iii). (iv) follows directly from (\ref{FejerSin}).
\end{proof}

Next lemma is also due to Bourgain \cite{B89} but with more complicated techniques. Here we use Salmensuu's \cite{S20} version and give the necessary details.
\begin{lemma}\label{BourCir}Let $0 < \kappa <1$, $u_0 > 2/\kappa$ and $u_1>0$. Suppose that $f, \nu:\,[N]\to\R$ are functions with $|f|\leq \nu$. Let us define the following Hardy-Littlewood decomposition:
\begin{align*}
&\M_{a,q} = \left\{\theta: \left|\theta - \frac{a}{q}\right| \leq \frac{1}{qQ} \right\} \\
&\M = \bigcup_{1 \leq q \leq P}\bigcup_{\substack{a=0 \\ (a,q)=1 }}^{q-1}\M_{a,q}, \quad \text{and} \quad \m = \T \setminus \M,
\end{align*}
where $P$, $Q$, $K$ are some positive numbers with  $P > C + K^{2/(\kappa u_1) + 1}$ for some large constant $C$, $2 P^2<Q<N/100$. 

Assume the following assumptions hold:
\begin{enumerate}[(i)]
\item $\int_{\T}|\widehat{f}(\theta)|^{u_0} d \theta \ll K N^{u_0-1}$.
    \item For any $\theta \in \M_{a,q}$ with $(a,q)=1$ and $1 \leq q \leq P$, we have $$\widehat{\nu}(\theta) \ll \frac{q^{-\kappa}N}{1 + N \|\theta-a/q\|} + o(NK^{-\frac{2}{u_1}} ).$$ \label{MjPriBour}
    \item For any $\theta \in \m$, we have $\widehat{\nu}(\theta) = o(NK^{-\frac{2}{u_1}})$.
\end{enumerate}
Then for any $u > u_0 + u_1$
$$
\int_{\T}|\widehat{f}(\theta)|^u d\theta\ll_u N^{u-1}.
$$
\end{lemma}
 
\begin{proof}
Let 
$$
\cR_{\delta}:=\{\theta \in \T: |\widehat{f}(\theta)|>\delta N\}.
$$
Arguing as in the proof of Lemma \ref{BourInt}, it is enough to prove that
$$
\mes(\cR_\delta)\ll\frac{1}{\delta^{u_0+u_1}N}
$$
for every 
\begin{equation}\label{deltaK}
\delta \geq K^{-\frac{1}{u_1}}.  
\end{equation}
Suppose that $\theta_1,\ldots,\theta_R\in \cR_\delta$  are $N^{-1}$-spaced and write $\theta_{r,r'} = \theta_r - \theta_{r'}$. For any $\gamma > 1$, by (\ref{Dual}) and H\"older's inequality we have 
\begin{equation}\label{RNnuGamma}
2^{-\gamma}R^{2} \delta^{2 \gamma} N^{2 \gamma} \leq R^{2-2 \gamma}\left(N \sum_{1 \leq r,r'\leq R} |\widehat{\nu}(\theta_{r,r'})|\right)^{\gamma} \leq N^{\gamma} \sum_{1 \leq r, r' \leq R}|\widehat{\nu}(\theta_{r,r'})|^{\gamma}.   
\end{equation}
Recalling $\delta \geq K^{-\frac{1}{u_1}}$, we have by assumption (iii)
\begin{equation}\label{minorsumnu}
 \sum_{\substack{1 \leq r, r' \leq R \\ \theta_{r,r'} \in \m}}|\hat{\nu}(\theta_{r,r'})|^{\gamma} = o(R^{2} \delta^{2 \gamma} N^{ \gamma}).
\end{equation}
Let $Q' = C + \delta^{-h}$, where $2/\kappa < h < 2/\kappa + u_1$. Recall (\ref{deltaK}). Note that those $\mathfrak{M}_{a,q}$ are disjoint, so from (\ref{MjPriBour}), we have that
\begin{align}    
& \sum_{Q'<q\leq P}\sum_{\substack{a=1 \\ (a,q)=1}}^{q}\sum_{\substack{1 \leq r,r' \leq R \\ \theta_{r,r'} \in \M_{a,q}}} |\hat{\nu}(\theta_{r,r'})|^{\gamma}\notag\\
& \ll \sum_{1 \leq r,r' \leq R } |\hat{\nu}(\theta_{r,r'})|^{\gamma} \notag\\
&\ll R^2 (Q')^{-\kappa \gamma}N^{\gamma} + o(R^2 \delta^{2 \gamma} N^{\gamma})\notag \\
& = R^2(C+\delta^{-h})^{-\kappa \gamma}N^{\gamma} + o(R^2 \delta^{2 \gamma} N^{\gamma}) \notag\\
& \leq  2\min\{C^{-\kappa \gamma}, \delta^{h \kappa \gamma}\}R^2 N^{\gamma} + o(R^2 \delta^{2 \gamma} N^{\gamma})\notag \\
& \leq 2\min\{C^{-\kappa \gamma}, \delta^{2 \gamma + (h-\frac{2}{\kappa})\kappa \gamma}\}R^2N^{\gamma} + o(R^2 \delta^{2 \gamma} N^{\gamma}). \label{q>Qmajor}
\end{align}
Since $C$ is a sufficiently large constant, (\ref{q>Qmajor}) can be bounded by $3^{-\gamma} R^2 \delta^{2 \gamma} N^{\gamma}$. By (\ref{RNnuGamma}), (\ref{minorsumnu}) and (\ref{q>Qmajor}), we have 
\begin{equation}\label{q<Q'case}
R^{2} \delta^{2 \gamma} N^{ \gamma} \ll \sum_{\substack{1\leq a\leq q\leq Q'\\ (a,q)=1}}
\sum_{\substack{1\leq r,r'\leq R\\
\theta_{r,r'}\in\M_{a,q}
}}|\hat{\nu}(\theta_{r,r'})|^{\gamma} \ll_\kappa \sum_{\substack{1\leq a\leq q\leq Q'\\ (a,q)=1}}
\sum_{\substack{1\leq r,r'\leq R\\
\theta_{r,r'}\in\M_{a,q}
}}\frac{N^\gamma q^{-\kappa \gamma }}{(1+N\|\theta_{r,r'}-\frac aq\|)^\gamma}.    
\end{equation}
Let
$$
F(\theta)=\frac{1}{(1+N\|\theta\|)^\gamma}
$$
and
$$
G(\theta):=\sum_{1\leq a\leq q\leq Q'}
q^{- \kappa \gamma}\cdot F\bigg(\theta-\frac aq\bigg).
$$
By (\ref{q<Q'case}), 
\begin{equation}\label{delta2gammaR2Gtheta}
\delta^{2\gamma}R^2\ll
\sum_{1\leq r,r'\leq R}G(\theta_{r,r'}).
\end{equation}
If $\|\theta-\theta'\|\leq N^{-1}$, then
$$
F(\theta)\asymp F(\theta')\quad \text{and} \quad G(\theta)\asymp G(\theta').
$$
Let $\sigma: \T \to \R$ be a function fulfilling four conditions in Lemma \ref{defsig}. 
Let 
$$\kappa(\theta)=\begin{cases}
1,&\text{if }\theta\in(-\frac1{10N},\frac1{10N}),\\
0,&\text{others}.
\end{cases}
$$
Clearly
$$
\sigma(\theta) \geq \sum_{r=1}^R\kappa(\theta-\theta_r).
$$
Let
$$
\sigma_1(\theta):=\sigma(-\theta).
$$
Then
\begin{align*}
(\sigma*\sigma_1)(\theta)=&\int_{\T}\sigma(t)\sigma(t-\theta)dt\\
\geq &\int_{\T}\bigg(\sum_{r=1}^R\kappa(t-\theta_r)\bigg)\cdot
\bigg(\sum_{r=1}^R\kappa(t-\theta-\theta_r)\bigg)d t\\
=&\sum_{1\leq r,r'\leq R}\int_{\T}\kappa(t-\theta_r)\kappa(t-\theta-\theta_{r'})d t\\
\gg &\frac{1}{N}\sum_{1\leq r,r'\leq R}\kappa(\theta-\theta_{r,r'}).
\end{align*}
Therefore,
\begin{align*}
\langle G,\sigma*\sigma_1\rangle=&\int_{\T} G(\theta)\cdot(\sigma*\sigma_1)(\theta) d\theta\\
\gg&\frac1N\sum_{1\leq r,r'\leq R}\int_{\T} G(\theta)\cdot\kappa(\theta-\theta_{r,r'})d\theta.
\end{align*}
Recall that $G(\theta)\asymp G(\theta_{r,r'})$ if $\kappa(\theta-\theta_{r,r'})=1$.
Hence, by (\ref{delta2gammaR2Gtheta}),
\begin{equation}\label{Gsigmasigma1a}
\langle G,\sigma*\sigma_1\rangle\gg
\frac1N\sum_{1\leq r,r'\leq R}G(\theta_{r,r'})
\int_{\T} \kappa(\theta-\theta_{r,r'})d\theta\gg
\frac{\delta^{2\gamma}R^2}{N^2}.
\end{equation}

On the other hand, we have
\begin{align*}
\widetilde{G}(k)=&\int_{\T} G(\theta) e(k\theta)d\theta
=
\sum_{q\leq Q'}\sum_{a=0}^{q-1}q^{- \kappa \gamma}\int_{\T}
F\bigg(\theta-\frac aq\bigg)e(k\theta)d\theta\\
=&
\sum_{q\leq Q'}q^{-\kappa \gamma}\int_{\T}
F(\theta)e(k\theta)\bigg(\sum_{a=0}^{q-1}e\bigg(\frac{ka}{q}\bigg)\bigg)d\theta=\widetilde{F}(k)\sum_{\substack{q\leq Q'\\ q\mid k}}q^{1-\kappa \gamma}.
\end{align*}
Assume that $\kappa \gamma >1$ and write
$$
d(k;Q'):=|\{q\leq Q':\,q\mid k\}|.
$$
Hence
$$
|\widetilde{G}(k)|\ll |\widetilde{F}(k)|\cdot d(k;Q')\leq \|F\|_{1}\cdot d(k;Q')\ll\frac{d(k;Q')}{N}.
$$
 
Recall that $\widetilde{\sigma}$ is supported on $[-N,N]$.  Let $0< \tau< 1 $ chosen later. By Parseval's identity
\begin{align*}
& \langle G,\sigma*\sigma_1\rangle=\sum_{|k|\leq N}\widetilde{G}(k)\overline{(\widetilde{\sigma*\sigma_1)}(k)} \leq \sum_{|k|\leq N}|\widetilde{G}(k)|\cdot|\hat{\sigma}(k)|^2
\\ 
\ll & \frac1N\sum_{|k|\leq N}|\widetilde{\sigma}(k)|^2\cdot d(k;Q') \leq \frac1N\bigg(Q'^\tau\sum_{\substack{|k|\leq N\\ d(k;Q')\leq Q'^\tau}}|\widetilde{\sigma}(k)|^2+|\widetilde{\sigma}(0)|^2\sum_{\substack{|k|\leq N\\ d(k;Q')\geq Q'^\tau}}d(k;Q')\bigg).
\end{align*}
Since $\sigma$ is bounded and $|\widetilde{\sigma}(0)|=\|\sigma\|_{1}\ll R/N$, we have by Parseval's identity
$$
\sum_{\substack{|k|\leq N\\ d(k;Q')\leq Q'^\tau}}|\widetilde{\sigma}(k)|^2\leq
\sum_{\substack{|k|\leq N}}|\widetilde{\sigma}(k)|^2
=\int_{\T}\sigma(\theta)^2d \theta\leq
\|\sigma\|_{\infty}\int_{\T}|\sigma(\theta)|d \theta\ll \frac{R}{N}.
$$
According to \cite[Lemma 4.28]{B89}, for any $B > \tau$, we have
$$
\sum_{\substack{|k|\leq N\\ d(k;Q')\geq Q'^{\tau}}}d(k;Q)\leq
Q'\sum_{\substack{|k|\leq N\\ d(k;Q')\geq Q'^{\tau}}}1\leq C_{\tau,B}Q'^{1-B}N.
$$
Consequently, 
\begin{equation}\label{Gsigmasigma1b}
\frac{\delta^{2\gamma}R^2}{N^2} \ll \langle G,\sigma*\sigma_1\rangle\ll \frac{Q'^{\tau} R}{N^2}+
\frac{Q'^{1-B} R^2}{N^{2}}.
\end{equation}
If we choose $B$ sufficiently large depending on $\gamma$, then the second term of the right hand side of (\ref{Gsigmasigma1b}) is negligible. 
Let $\gamma \in (1/\kappa, u_0/2) $ and $\tau \leq u_1/h$. We obtain that
$$
R\ll \delta^{-2\gamma}Q'^{\tau}\ll \delta^{-2\gamma}(C+\delta^{-h})^{\tau} \ll \delta^{-2 \gamma }\delta^{-\tau h} \ll \delta^{-u_0 -u_1}.
$$
\end{proof}
Now we can prove Theorem \ref{primecWnbT}.
\begin{proof}[Proof of Theorem \ref{primecWnbT}] We shall apply Lemma \ref{BourCir} with $f= f_{W,X,b}$, $\nu(n) = \nu_{W,X,b}(n)$, $u_0$, $u_1$ fulfilling that $u> u_0 > 2+{4(c-1)}/{(2-c)}$ and $0 < u_1 < u- u_0$, and $K=\log^{u_0} N$. By Lemma \ref{falphaeta}, we have
\begin{equation}\label{BourCondiOne}
\int_{\T}|\widehat{f}(\theta)|^{u_0}d\theta \ll 
\bigg(\frac{\phi(W)}{W}\cdot 2\log N\bigg)^{u_0} N^{u_0 - 1} \ll (\log N)^{u_0} N^{u_0 - 1}= KN^{u_0-1},
\end{equation}
so Lemma \ref{BourCir} (i) holds. Lemma \ref{BourCir} (ii) and (iii) follow from Theorem \ref{primecWnbT} (i) and (ii). Consequently, Theorem \ref{primecWnbT} follows from Lemma \ref{BourCir}.
\end{proof}

\section{Lower bound sieve}\label{LBsieve}
\setcounter{lemma}{0} \setcounter{theorem}{0}
\setcounter{equation}{0}

In this section we will prove Theorem \ref{LUthm} (iii) holds for $c \in (1, 41/35)$ and $\alpha^{-} = 0.7816\dots$. By dyadic splitting we will utilize the well-known type I and II information for Piateski-Shapiro primes to construct a lower bound sieve for $\rho(n)$ by Harman's sieve method. However we need the type I and II information for arithmetic progressions, see  Theorem \ref{primecWnbT} (\ref{LowSieveInAps}), so we need to show arithmetic progression version of type I and II information. 

Corresponding to the notation in many papers related to Piatetski-Shapiro primes, we write $\gamma = \frac{1}{c}$ in the remaining sections.

\begin{lemma}[Type I sum for arithmetic progressions]\label{VanExpTypeI}
Let $A \geq 2$. Suppose that $\epsilon$ and $\delta$ are small positive numbers such that $\epsilon > 100 \delta$. Let $16/19 + \epsilon \leq \gamma <1, MN \asymp X, H \leq X^{1 - \gamma + 4 \delta}$. Assume further that $a(m),c(h)$ are complex numbers of modulus $\leq 1$, $d \leq \log^A X$, $(l,d)=1$ and $M$ satisfies the condition
\begin{equation}
    M < X^{3 \gamma - 2 - \epsilon}.
\end{equation}
Then 
\begin{equation}\label{LowTypeIap}
\sum_{h \sim H}\sum_{m \sim M}\sum_{\substack{n \sim N \\ mn \equiv l \pmod{d}}}c(h)a(m)e(h(mn)^{\gamma}) \ll X^{1 - 5\delta}.   
\end{equation}

\end{lemma}
\begin{proof}
For $k \in [d]$ with $(k,d)=1$, let $a'(m,k) =a(m) \1_{m \equiv k \pmod{d}}$ and let $j_k \in [d]$ be such that $j_k k \equiv l \pmod{d}$. By splitting $m$ in (\ref{LowTypeIap}) into $\phi(d)$ cases and writting $n = j_k+rd$, we have 
\begin{align*}
&\sum_{h \sim H}\sum_{m \sim M}\sum_{\substack{n \sim N \\ mn \equiv l \pmod{d}}}c(h)a(m)e(h(mn)^{\gamma})\\
= & \sum_{\substack{1 \leq k \leq d \\ (k,d)=1}}\sum_{h \sim H}\sum_{m \sim M}\sum_{r \sim \frac{N}{d}}c(h)a'(m,k)e(h(m(j_k+rd))^{\gamma}) + O(HM).
\end{align*}
The rest of the proof is same as the proof of \cite[Lemma 4]{K99}.
\end{proof}

\begin{lemma}[Type II sum for arithmetic progressions]\label{VanExpTypeII}     
Let $A \geq 2$. Suppose that $\epsilon$ and $\delta$ are small positive numbers such that $\epsilon > 100 \delta$. Let $16/19 + \epsilon \leq \gamma <1, MN \asymp X, H \leq X^{1 - \gamma + 4 \delta}$. Assume further that $a(m), b(n),c(h)$ are complex numbers of modulus $\leq 1$, $d \leq \log^A X$, $(l,d)=1$ and $M$ satisfies the condition
$$
     X^{1- \gamma + \epsilon} < M < X^{5 \gamma - 4 - \epsilon} \quad \text{ or } \quad   X^{5-5 \gamma + \epsilon} < M < X^{ \gamma - \epsilon}
$$
$$
   \text{or} \quad X^{3-3\gamma + \epsilon} < M < X^{3 \gamma - 2 - \epsilon}.
$$
Then 
\begin{equation}\label{LowTypeIIap}
\sum_{h \sim H}\sum_{m \sim M}\sum_{\substack{n \sim N \\ mn \equiv l \pmod{d}}}c(h)a(m)b(n)e(h(hmn)^{\gamma}) \ll X^{1 - 5\delta}.   
\end{equation}

\end{lemma}

\begin{proof}
    The left hand side of (\ref{LowTypeIIap}) equals 
    $$
    \frac{1}{\phi(d)}\sum_{\chi}\overline{\chi}(l)\sum_{h \sim H}\sum_{m \sim M}\sum_{n \sim N }c(h)a(m)\chi(m)b(n)\chi(n)e(h(hmn)^{\gamma}).
    $$
    We apply \cite[Lemmas 1 and 3]{K99} with $a(m)\chi(m)$ in place of $a(m)$ and $b(n)\chi(n)$ in place of $b(n)$. 
\end{proof}
Let $X$ be a sufficiently large positive number. Fox fixed $d \leq \log^A X$ with any $A>0$ and $(l,d)=1$, we define
$$
\cA = \{n \sim X: n \in \N^c, n \equiv l \pmod{d} \},
$$
and 
$$
\cB = \{n \sim X: n \equiv l \pmod{d}\}.
$$
By using the above two lemmas, we can get our type I and II information which will be used in Harman's sieve method.

\begin{lemma}\label{TypeIandIIPre}
Let $\epsilon>0$ be small, $16/19 + \epsilon \leq \gamma <1$, $0<\delta<\epsilon/{100}$ and $a(m),b(m) \ll d^{O(1)}(m)$. For 
$$
 M < X^{3 \gamma - 2 - \epsilon},
$$
we have
\begin{equation}\label{PreLowSieveTypeI}
\sum_{\substack{ mn \in \cA \\ m \sim M }}\frac{1}{\gamma}a(m) (mn)^{1-\gamma} = \sum_{\substack{ mn \in \cB \\m \sim M }}a(m) + O(X^{1 - 3 \delta}).    
\end{equation}

Moreover if 
$$
     X^{1- \gamma + \epsilon} < M < X^{5 \gamma - 4 - \epsilon} \quad \text{ or } \quad   X^{5-5 \gamma + \epsilon} < M < X^{ \gamma - \epsilon}
$$
$$
   \text{or} \quad X^{3-3\gamma + \epsilon} < M < X^{3 \gamma - 2 - \epsilon},
$$
then
$$
\sum_{\substack{mn \in \cA \\m \sim M }}\frac{1}{\gamma}a(m)b(n) (mn)^{1-\gamma}= \sum_{\substack{ mn \in \cB \\ m \sim M}}a(m)b(n) + O(X^{1- 3 \delta}).
$$
\end{lemma}
\begin{proof}
Let us consider type II sum. We have
$$
 \sum_{\substack{m \sim M \\ mn \in \cA}}\frac{1}{\gamma}a(m)b(n) (mn)^{1-\gamma} = \sum_{\substack{m \sim M \\ mn \in \cB}}\frac{1}{\gamma}a(m)b(n) (mn)^{1-\gamma}( \lfloor- (mn)^{\gamma}\rfloor - \lfloor- (mn+1)^{\gamma}\rfloor) = \Sigma_1 + \Sigma_2,
$$
where
$$
\Sigma_1 = \sum_{\substack{m \sim M \\ mn \in \cB}}\frac{1}{\gamma}a(m)b(n) (mn)^{1-\gamma}((mn+1)^{\gamma} - (mn)^{\gamma})
$$ 
and
$$
\Sigma_2 = \sum_{\substack{m \sim M \\ mn \in \cB}}\frac{1}{\gamma}a(m)b(n) (mn)^{1-\gamma}(-\{-(mn)^{\gamma}\} + \{-(mn+1)^{\gamma}\} ).
$$
Note that
$$
\Sigma_1 = \sum_{\substack{m \sim M \\ mn \in \cB}}a(m)b(n) + O(X^{\delta}).
$$
Next we apply well-known reduction argument using the Fourier expansion of the function $\psi(t)$, namely Vaaler's theorem (see  e.g. \cite[Theorem A.6]{GK91}). According to it, for any $J \geq 1$, there exist functions $\psi_1$ and $\psi_2$ such that
$$
\psi(t) = \psi_1(t) + O(\psi_2(t)),
$$
where 
$$
\psi_1(t) = \sum_{1 \leq |j| \leq J} \beta_1(j)e(jt) \quad \text{and} \quad \psi_2(t) = \sum_{1 \leq |j| \leq J} \beta_2(j)e(jt),
$$
and $\beta_1(j) \ll j^{-1}, \quad \beta_2(j) \ll J^{-1}$. Moreover, $\psi_2$ is non-negative. Let 
$$
f(k) = \mathbf{1}_{k \equiv l \pmod{d}}\frac{1}{\gamma}k^{1-\gamma}\sum_{\substack{  mn = k \\ m\sim M}}a(m)b(n),
$$
and note that $f(k) \leq  X^{1-\gamma + \delta}$.
Consequently,
\begin{align*}
\Sigma_2 & =  \sum_{k \in \cB}f(k)( \psi_1(-(k+1)^{\gamma}) -\psi_1(-k^{\gamma}) ) + O(\sum_{k \in \cB}|f(k)|(\psi_2(-k^{\gamma}) + \psi_2(-(k+1)^{\gamma})))\\
& =: \Sigma_3 + O(\Sigma_4),
\end{align*}
say. We take $J = X^{1-\gamma + 4 \delta}$. Let us first consider $\Sigma_4$. By Lemma \ref{vanderCorput}, we have 
\begin{align*}
\sum_{k \sim X}\psi_2(-k^{\gamma}) & \ll \frac{1}{J} \sum_{1 \leq |j| \leq J} \left| \sum_{k \sim X}e(jk^{\gamma}) \right|  \ll J^{-1}\sum_{1 \leq |j| \leq J} \left(X(jX^{\gamma -2})^{\frac{1}{2}} + (jX^{\gamma -2})^{-\frac{1}{2}} \right)\\
 \ll & J^{\frac{1}{2}}X^{\frac{\gamma}{2}} + J^{-\frac{1}{2}}X^{1-\frac{\gamma}{2}} \ll X^{\gamma - 4\delta}.
\end{align*}
The same estimate holds if $k$ is replaced by $k+1$. Hence 
$$
\Sigma_4 \ll X^{1 - 3 \delta}.
$$
To complete the proof, it suffices to show that $\Sigma_3 \ll X^{1 - 3 \delta}$. We write
$$
g_j(t) = 1 - e(j(t^{\gamma} - (t+1)^{\gamma})).
$$
By partial summation and Lemma \ref{VanExpTypeII}, we have
\begin{align*}
& \Sigma_3 = \sum_{1 \leq |j| \leq J}\beta_1(j)\sum_{k \sim X}f(k)(e(-jk^{\gamma}) - e(-j(k+1)^{\gamma}))\\
\ll & \sum_{1 \leq j \leq J}\beta_1(j) \left|\sum_{k \sim X}f(k)g_j(k)e(-jk^{\gamma}) \right|\\
\ll & \sum_{1 \leq j \leq J}j^{-1}\left| g_j(2X)(2X)^{1-\gamma} \sum_{k \sim X}\frac{f(k)}{k^{1 - \gamma}}e(-jk^{\gamma}) \right|\\
+ & \int_{X}^{2X} \sum_{1 \leq j \leq J}j^{-1}\left| \frac{d}{du}(g_{j}(u)u^{1-\gamma}) \sum_{X \leq k \leq u}\frac{f(k)}{k^{1 - \gamma}}e(-jk^{\gamma})\right|du\\
\ll & \max_{X \leq u \leq 2X} \sum_{1 \leq j \leq J}\left|\sum_{X \leq k \leq u}\frac{f(k)}{k^{1 - \gamma}}e(jk^{\gamma})\right|.
\end{align*}
Hence for some $c(j) \in \C$ with with $|c(j)| = 1$.
$$
\Sigma_3 \ll \max_{X \leq u \leq 2X} \sum_{1 \leq j \leq J}c(j)\sum_{\substack{X \leq mn \leq u \\ mn \equiv l \pmod{d} \\ m \sim M}}a(m)b(n)e(j(mn)^{\gamma}).
$$
By using \cite[Lemma 2.2]{Ha07} one can remove the crossing condition $X \leq mn \leq u$ and the claim follows from Lemma \ref{VanExpTypeII}. 
One can prove (\ref{PreLowSieveTypeI}) similarly by using Lemma \ref{VanExpTypeI} instead of Lemma \ref{VanExpTypeII}.
\end{proof}

Recall $\rho(n,z) = \1_{p \mid n \Rightarrow p \geq z}$. We define, for $m \in \N$,
$$
S(\cA_m,z) = \frac{1}{\gamma}\sum_{mn \in \cA}\rho(n,z)(mn)^{1-\gamma} \quad \text{and} \quad S(\cB_m,z) = \sum_{mn \in \cB}\rho(n,z).
$$
\begin{lemma}\label{SieveLemma}
Let $u \geq 1$, $\epsilon>0$ be a small positive number. Let $16/19 + \epsilon \leq \gamma <1$ and $0< \delta < \epsilon/{100}$. 
\begin{enumerate}[(i)]
    \item 
    Let $P_1, \dots, P_u \in [1,X]$ be such that $\prod_{1 \leq k \leq u}P_k \leq X$ and $z \in [x^\epsilon, X^{6 \gamma -5 -2\epsilon}]$. Suppose that for some $M$ satisfying
$$
     X^{1- \gamma + \epsilon} < M < X^{5 \gamma - 4 - \epsilon} \quad \text{ or } \quad   X^{5-5 \gamma + \epsilon} < M < X^{ \gamma - \epsilon}
$$
$$
   \text{or} \quad X^{3-3\gamma + \epsilon} < M < X^{3 \gamma - 2 - \epsilon},
$$
there exists $\cD \subset \{1, \dots, u\}$ such that  
$$
\prod_{k \in \cD}P_k \asymp M.
$$
Then 
\begin{equation}\label{SieveTypeII}
\sum_{\substack{z \leq p_u < \dots < p_1 \leq (2X)^{1/2} \\ p_i \sim P_i}}S(\cA_{p_1 \cdots p_u}, p_u) = \sum_{\substack{z \leq p_u < \dots < p_1 \leq (2X)^{\frac{1}2} \\ p_i \sim P_i}}S(\cB_{p_1 \cdots p_u}, p_u) + O(X^{1 - \delta}).    
\end{equation}
\item If
$$
 M < X^{3 \gamma - 2 - \epsilon},
$$
then, for any $a(m) \ll d^{O(1)}(m)$,

\begin{equation}\label{SieveTypeI/II}
\sum_{m \sim M}a(m)S(\cA_m, z) =  \sum_{m \sim M}a(m)S(\cB_m, z) + O(X^{1 - \delta}).   
\end{equation}
\end{enumerate}
\end{lemma}
\begin{proof}
Apart from removing cross-conditions, (\ref{SieveTypeII}) is immediately from Lemma \ref{TypeIandIIPre}. On the other hand (\ref{SieveTypeI/II}) follows from Lemma \ref{TypeIandIIPre} and Harman's sieve method \cite[Theorem 3.1]{Ha07} with $M < X^{3 \gamma -2 - \epsilon}, X^\theta = X^{3-3 \gamma + \epsilon}$ and $X^\beta = X^{6 \gamma - 5 - 2 \epsilon}$.
For more details, one can refer to \cite[Lemmas 9 and 10]{K99}. Note that in \cite[Lemmas 9 and 10]{K99}, if $\gamma = 38/45 + \epsilon$, then $z = X^{\frac{1}{15}} = X^{6 (\gamma-\epsilon) - 5} $ corresponding our choice of $z$ by ignoring $\epsilon$. The only difference is that one needs to replace type I and II information by our Lemma \ref{TypeIandIIPre}.  
\end{proof}
Our goal now is to construct a lower bound sieve weight by applying the above lemma.
For $\gamma \in (35/41,1)$, let $\gamma \geq {35}/{41} + \epsilon_1$ for some $\epsilon_1>0$, then apply Lemma \ref{SieveLemma} with $\epsilon = \epsilon_1$. We choose $z= X^{\frac{5}{41}}$ and type I range includes
$$
\Omega_1 = [1, X^{\frac{23}{41}}],
$$
and type II range includes
$$
\Omega_2 = [X^{\frac{6}{41}}, X^{\frac{11}{41}}] \cup [X^{\frac{18}{41}}, X^{\frac{23}{41}}] \cup [X^{\frac{30}{41}}, X^{\frac{35}{41}}].
$$

We first apply Buchstab's identity twice and get

\begin{align*}
S(\cA,(2X)^{\frac{1}2}) & = S(\cA,z) - \sum_{z \leq p < (2X)^{\frac{1}2}} S(\cA_p, z) + \sum_{z \leq p_2 <  p_1 < (2X)^{\frac{1}2}}S(\cA_{p_1p_2},p_2)\\
= S_1 - S_2 + S_3.
\end{align*}
For $S_1,S_2$, we use Lemma \ref{SieveLemma}. For $S_3$ we give further decomposition
\begin{align*}
S_3 & = \sum_{\substack{z < p_2 <  p_1 \leq (2X)^{\frac{1}2} \\ \{p_1, p_2, p_1p_2\} \cap \Omega_2 \neq \emptyset}}S(\cA_{p_1p_2},p_2) + \sum_{\substack{z \leq p_2 <  p_1 \leq (2X)^{1/2} \\ \{p_1, p_2, p_1p_2\} \cap \Omega_2 = \emptyset}}S(\cA_{p_1p_2},p_2) = S_4 + S_5.
\end{align*}
For $S_4$ we use Lemma \ref{SieveLemma}. Note that $S(\cA_{p_1p_2},p_2)$ in $S_5$ is zero unless $p_1p_2 \in [X^{\frac{11}{41}}, X^{\frac{18}{41}}] \cup [X^{\frac{23}{41}}, X^{\frac{30}{41}}]$. We decompose
\begin{align*}
S_5 & =\sum_{\substack{z \leq p_2 <  p_1 < X^{\frac{6}{41}}\\ p_1p_2 > X^{\frac{11}{41}}}}S(\cA_{p_1p_2},p_2) + \sum_{\substack{z < p_2 < X^{\frac{6}{41}} \\ X^{\frac{11}{41}} < p_1 <X^{\frac{18}{41}} \\ p_1p_2 < X^{\frac{18}{41}}}}S(\cA_{p_1p_2},p_2)\\
& + \sum_{\substack{z \leq p_2 < X^{\frac{6}{41}} \\ X^{\frac{11}{41}} < p_1 < X^{\frac{18}{41}}\\ p_1p_2 > X^{\frac{23}{41}}}}S(\cA_{p_1p_2},p_2) + \sum_{\substack{X^{\frac{11}{41}} < p_2 < p_1 <X^{\frac{18}{41}} \\ X^{\frac{23}{41}}< p_1p_2 < X^{\frac{30}{41}}}}S(\cA_{p_1p_2},p_2) \\
& = S_6 + S_7 + S_8 + S_9.
\end{align*}

For $S_6$, we apply Buchstab's decomposition twice more. In $S_7$, $p_1 < X^{\frac{13}{41}}$. For the part of $S_7$ with $p_1 p^2_2 \in \Omega_1$, we can apply Buchstab's decomposition once more. For $S_8$, we simply discard it, i.e., use that $S_8 \geq 0$, since Lemma \ref{SieveLemma} cannot be applied, and note that necessarily $p_1 > X^{\frac{17}{41}}$. For $S_9$, we add a restrictive condition $p_1p^2_2 \leq 2X$ and then discard it. 

Hence, we have 
\begin{align*}
S_{6} & = \sum_{\substack{z \leq p_2 <  p_1 < X^{\frac{6}{41}}\\ p_1p_2 > X^{\frac{11}{41}}}}S(\cA_{p_1p_2},z) - \sum_{\substack{z \leq p_3 < p_2 <  p_1 < X^{\frac{6}{41}}\\ p_1p_2 > X^{\frac{11}{41}}}}S(\cA_{p_1p_2p_3},z)\\
& + \sum_{\substack{z \leq p_4 < p_3 < p_2 < p_1 < X^{\frac{6}{41}}\\ p_1p_2 > X^{\frac{11}{41}} \\ p_1p_2p_3p_4\leq  X^{\frac{23}{41}}}}S(\cA_{p_1p_2p_3p_4},p_4) +  \sum_{\substack{z \leq p_4 < p_3 < p_2 <  p_1 < X^{\frac{6}{41}}\\ p_1p_2 > X^{\frac{11}{41}} \\ p_1p_2p_3p_4 > X^{\frac{23}{41}}}}S(\cA_{p_1p_2p_3p_4},p_4)\\
& = S_{10} - S_{11} + S_{12} + S_{13},
\end{align*}
and
\begin{align*}
S_7 & = \sum_{\substack{z \leq p_2 < X^{\frac{6}{41}} \\ X^{\frac{11}{41}} < p_1 <X^{\frac{13}{41}} \\ p_1p^2_2 \in \Omega_1}}S(\cA_{p_1p_2},z) - \sum_{\substack{z \leq p_3 < p_2 < X^{\frac{6}{41}} \\ X^{\frac{11}{41}} < p_1 <X^{\frac{13}{41}} \\ p_1p^2_2 \in \Omega_1}}S(\cA_{p_1p_2p_3},p_3)\\
& + \sum_{\substack{z < p_2 < X^{\frac{6}{41}} \\ X^{\frac{11}{41}} < p_1 <X^{\frac{13}{41}} \\ p_1p_2 < X^{\frac{18}{41}} \\ p_1p^2_2 \not\in \Omega_1}}S(\cA_{p_1p_2},p_2)\\
& = S_{14} - S_{15} + S_{16}.
\end{align*}

In case of $S_{12}$ necessarily $p_1p_2p_3p_4 \in \Omega_2$ and in case of $S_{15}$ necessarily $p_1p_2p_3 \in \Omega_2$, so for $S_{10}, S_{11}, S_{12}, S_{14}$ and $S_{15}$, we apply Lemma \ref{SieveLemma}. For $S_{13}$ and $S_{16}$, we simply discard them.

Note that all terms we have discarded give positive contribution to $S(\cA, (2X)^{\frac{1}2})$. Let
$$
S^{-}(\cA) =  S(\cA, (2X)^{\frac12}) - S_8 - S_9 - S_{13} - S_{16},
$$ 
which is a lower bound of $S(\cA, (2X)^{\frac12})$. We do the corresponding decomposition for $S(\cB, (2X)^{\frac12})$ and we get that
$$
S^{-}(\cA) =  S(\cB, (2X)^{\frac12}) - T_8 - T_9 - T_{13} - T_{16} + O(X^{1 - \delta}), 
$$
where $T_8,T_9,T_{13}$ and $T_{16}$ are almost the same as the discarded sums with the only difference that $\cA$ is replaced by $\cB$, for example,
$$
T_{13} = \sum_{\substack{z \leq p_4 < p_3 < p_2 <  p_1 < X^{\frac{6}{41}}\\ p_1p_2 > X^{\frac{11}{41}} \\ p_1p_2p_3p_4 > X^{\frac{23}{41}}}}S(\cB_{p_1p_2p_3p_4},p_4).
$$
The remaining task is to calculate the contributions of $T_8, T_9, T_{13}$ and $T_{16}$. The calculation is standard and we refer interested readers to \cite{Ha07} or \cite{K99}. Here we just calculate $T_{13}$ as an example.
Let us first define the corresponding region $\cD_{13}$ to be
$$
\left\{(\theta_1, \theta_2,\theta_3,\theta_4): \frac{5}{41} < \theta_4 \leq \theta_3 \leq \theta_2 \leq \theta_1 \leq \frac{6}{41}, \theta_1 + \theta_2 > \frac{11}{41}, \theta_1 + \theta_2 +\theta_3 + \theta_4  > \frac{23}{41} \right\}.
$$

By the Siegel-Walfisz theorem, for any sufficiently large positive real number $A$, 
$$
T_{13} = I_{13} S(\cB, (2X)^{1/2}) + O(X/(\log X)^A),
$$
where
$$
I_{13} = \iiiint_{\cD_{13}}\omega\left(\frac{1-\theta_1 - \theta_2 - \theta_3 - \theta_4}{\theta_4}\right) \frac{1}{ \theta_1\theta_2\theta_3 \theta^2_4}d\theta_1 d\theta_2 d\theta_3 d\theta_4,
$$
and $\omega(x)$ is Buchstab function, see \cite[Chapter 1]{Ha07}.
With the help of Mathematica program, we have
$I_{13} < 0.0001$. Arguing similarly with $T_8, T_9$ and $T_{16}$, we obtain
$$
S(\cA,(2X)^{\frac12}) = (1- I_8-I_9-I_{13}-I_{16} )S(\cB,(2x)^{\frac12}) + O(X(\log X)^{-A}),
$$
where $I_8 < 0.0206, I_9 < 0.1675$ and $I_{16} < 0.0302$.
Thus
$$
\sum_{\substack{n \sim X \\ n \in \cP \cap \N^c \\ n \equiv l \pmod{d}}}c n^{1-\frac{1}{c}} \rho(n) = S(\cA,(2X)^{\frac12}) > 0.7816 \cdot S(\cB,(2x)^{\frac12}).
$$
Consequently, Theorem \ref{LUthm} (iii) holds for $c \in (1, \frac{41}{35})$ and $\alpha^{-} =  0.7816$.

\begin{Rem}\label{rem3}
We believe one can only slightly improve our the upper bound $41/35$ without new ideas. The reason is that when $c$ becomes larger the lower bound sieve coefficient $\alpha^{-}$ becomes smaller and the upper bound sieve coefficient $\alpha^{+}$ becomes larger. Thus it is hard to satisfy the condition $3 \alpha^{-} - \alpha^{+} >0$ when $c$ is large.
\end{Rem}

\section{Upper bound Sieve}\label{UBsieve}
\setcounter{lemma}{0} \setcounter{theorem}{0}
\setcounter{equation}{0}
In this section, we will construct $\rho^{+}(n)$ in Theorem \ref{LUthm}. We will first prove that Theorem \ref{LUthm} (iv) holds in the end of this section, and will prove that Theorem \ref{LUthm} (i) and (ii) hold in Section \ref{UBFourier}.

For constructing a suitable upper bound sieve, we need type I and II information for Fourier transform of $\rho^{+}(n)$, namely $\rho^{+}(n)e(n \theta)$ for $\theta \in \T$, see condition (\ref{PsedoRandom}) in Theorem \ref{LUthm}. However, in this case, one can only obtain a relative narrow type I and II ranges compared to the lower bound sieve case. The reason, roughly speaking, is that for $\theta = 0$, one can use Fouvry-Iwaniec method to enlarge type I and II ranges, (see \cite{K97}), but for general $\theta$ this method is invalid. In this section, we will use type I and II information for $\rho ^{+}(n)e(n \theta)$ due to Balog and Friedlander work \cite{BF92}.

Let $X$ be a sufficiently large positive number.
We define
$$
\cA = \{n \sim X: n \in \N^c \},
\quad
\text{and} 
\quad
\cB = \{n \in \Z: n \sim X \}.
$$
It suffices to first consider Type I and II without conditions like $n \equiv b \pmod{q}$, since it will be easy to add such conditions using the orthogonality of additive characters; see Lemma \ref{ApUpp}.
\begin{lemma}[Type I]\label{upperTypeI}
Suppose that $5/6 + \epsilon \leq \gamma < 1, \theta \in \T, MN \asymp X,H \leq X^{1-\gamma+ \epsilon + \delta}$, and assume further that $a(m), c(h)$ are complex numbers of modulus $\leq 1$, and
$$
M < X^{ 4 \gamma -3 - \epsilon}.
$$
Then 
$$
\sum_{h \sim H}\sum_{m \sim M}\sum_{\substack{n \sim N \\ mn \in \cB }}c(h)a(m)e(\theta(mn) + h(mn)^{\gamma}) \ll X^{1 - \epsilon -\delta}.
$$
Furthermore, we have
$$
\frac{1}{\gamma}\sum_{\substack{m \sim M \\ mn \in \cA}}a(m)(mn)^{1-\gamma}e(\theta mn) = \sum_{\substack{m \sim M \\ mn \in \cB}}a(m)e(\theta mn) + O(X^{1 -  \epsilon}).
$$
\end{lemma}
\begin{proof}
See proof of \cite[Proposition 3]{BF92}.
\end{proof}

\begin{lemma}[Type II]\label{upperTypeII}
Suppose that $5/6 + \epsilon \leq \gamma < 1, \theta \in \T, MN \asymp X,H \leq X^{1-\gamma+ \epsilon + \delta}$, and assume further that $a(m), b(n), c(h)$ are complex numbers of modulus $\leq 1$,  and either 
$$
     X^{1- \gamma + \epsilon} < M < X^{5 \gamma - 4 - \epsilon} \quad \text{ or } \quad   X^{5-5 \gamma + \epsilon} < M < X^{ \gamma - \epsilon}.
$$
Then 
$$
\sum_{h \sim H}\sum_{m \sim M}\sum_{\substack{n \sim N \\ mn \in \cB}}c(h)a(m)b(n)e(\theta(mn) + h(mn)^{\gamma}) \ll X^{1 - \epsilon -\delta}.
$$
Furthermore, we have 
$$
\frac{1}{\gamma}\sum_{\substack{m \sim M \\ mn \in \cA}}a(m)b(n)(mn)^{1-\gamma}e(\theta mn) = \sum_{\substack{m \sim M \\ mn \in \cB}}a(m)b(n)e(\theta mn) + O(X^{1 - \epsilon}).
$$
\end{lemma}
\begin{proof}
See proof of \cite[Proposition 2]{BF92}.
\end{proof} 

For any $\theta \in \T$, we define
$$
S_{\theta}(\cA_m,z) = \frac{1}{\gamma}\sum_{mn \in \cA}\rho(n,z)(mn)^{1-\gamma}e(mn \theta)
\quad
\text{and}
\quad
S_{\theta}(\cB_m,z) = \sum_{mn \in \cB}\rho(n,z)e(mn \theta).
$$
If $\theta = 0$, we define
$$
S(\cA_m,z) = \frac{1}{\gamma}\sum_{mn \in \cA}\rho(n,z)(mn)^{1-\gamma}
\quad
\text{and}
\quad
S(\cB_m,z) = \sum_{mn \in \cB}\rho(n,z).
$$

\begin{lemma}\label{SieveLemmaExp}
Let $\theta \in \T$,  $u \geq 1$, $\epsilon>0$ be small and let $X>0$ be large and $z \in [x^\epsilon, X^{6 \gamma -5 -2\epsilon}]$.
\begin{enumerate}[(i)]
    \item Let
$P_1, \dots, P_u \in [1,X]$ be such that $\prod_{1 \leq k \leq u}P_k \leq X$ Suppose that for some $M$ satisfying
$$
     X^{1- \gamma + \epsilon} < M < X^{5 \gamma - 4 - \epsilon} \quad \text{ or } \quad   X^{5-5 \gamma + \epsilon} < M < X^{ \gamma - \epsilon},
$$
there exists $\cD \subset \{1, \dots, u\}$ such that  
$$
\prod_{k \in \cD}P_k \asymp M.
$$
Then 
$$
\sum_{\substack{z \leq p_u < \dots < p_1 < (2X)^{\frac12} \\ p_i \sim P_i}}S_{\theta}(\cA_{p_1,\cdots,p_u}, p_u) = \sum_{\substack{z \leq p_u < \dots < p_1 < (2X)^{\frac12} \\ p_i \sim P_i}}S_{\theta}(\cB_{p_1,\cdots,p_u}, p_u) + O(X^{1 - \epsilon}).
$$
\item If
$$
 M < X^{5 \gamma - 4 - \epsilon },
$$
then for $a(m) \ll d^{O(1)}(m)$

$$
\sum_{m \sim M}a(m)S_{\theta}(\cA_m, z) =  \sum_{m \sim M}a(m)S_{\theta}(\cB_m, z) + O(X^{1 - \epsilon}).
$$
\end{enumerate}
\end{lemma}
\begin{proof}
This follows similarly to the proof of Lemma \ref{SieveLemma}.
\end{proof}

Similarly to the construction of the lower bound sieve, for $\gamma \in (35/41,1)$, let $\gamma \geq 35/41 + \epsilon_1$ for some $\epsilon_1>0$. We can pick $z = X^{\frac5{41}}$ and type I range includes
$$
\Omega_1=[1,X^{\frac{17}{41}}]
$$ 
and type II range includes
$$
\Omega_2=[X^{\frac{6}{41}},X^{\frac{11}{41}}] \cup [X^{\frac{30}{41}},X^{\frac{35}{41}}].
$$
For the upper bound sieve, we apply Buchstab's decomposition once to get
\begin{align*}
    S_{\theta}(\cA,(2X)^{1/2}) & = S_{\theta}(\cA,z) - \sum_{z \leq p_1 < X^{\frac{6}{41}}}S_{\theta}(\cA_{p_1},p_1) \\
    & - \sum_{X^{\frac{6}{41}} \leq p_1 \leq X^{\frac{11}{41}} }S_{\theta}(\cA_{p_1},p_1) - \sum_{X^{\frac{11}{41}} < p_1 < (2X)^{\frac{1}{2}}}S_{\theta}(\cA_{p_1},p_1)\\
    & = S_1 - S_2 - S_3 - S_4.
\end{align*}
For $S_1$ and $S_3$ we apply Lemma \ref{SieveLemmaExp}. We simply discard $S_4$ and apply Buchstab's decomposition twice more to $S_2$ obtaining 
\begin{align*}
    S_2 & = \sum_{z <p_1 \leq X^{\frac{6}{41}}}S_{\theta}(\cA_{p_1},z) - \sum_{z \leq p_2 \leq p_1\leq X^{\frac{6}{41}}}S_{\theta}(\cA_{p_1p_2},z) + \sum_{z \leq p_3 < p_2 < p_1\leq X^{\frac{6}{41}}}S_{\theta}(\cA_{p_1p_2p_3},p_3)\\
    & = S_5 - S_6 + S_7.
\end{align*}
For $S_5$ we use Lemma \ref{SieveLemmaExp}. For $S_6$ we can also use Harman's sieve method, but in slightly different way then before. Namely, we write by M\"obius inversion 
\begin{align*}
\sum_{z \leq p_2 \leq p_1\leq X^{\frac{6}{41}}}S_{\theta}(\cA_{p_1p_2},z) & = \frac{1}{\gamma}\sum_{d \mid \cP(z)} \mu(d) \sum_{\substack{z \leq p_2 < p_1\leq X^{\frac{6}{41}} \\ p_1p_2dm \in \cA}}(p_1p_2dm)^{1-\gamma}e(\theta p_1p_2dm).
\end{align*}
\begin{itemize}
    \item If $p_2 d < X^{\frac{6}{41}}$, then $p_1p_2d \in \Omega_1$ and we have a type I sum;
    \item if $X^{\frac{6}{41}} < p_2 d <X^{\frac{11}{41}}$, then $p_2d \in \Omega_2$ and we have a type II sum;
    \item if $p_2 d > X^{\frac{11}{41}}$, we write $d = q_1 q_2 \cdots q_t$ with $z> q_1 > q_2 > \dots > q_t$. Now $p_2 q_1 \cdots q_{j-1} < X^{\frac{11}{41}} < p_2 q_1 \cdots q_{j} $ for some $j \leq t$. Since $q_1 \leq z$, we have $X^{\frac{6}{41}}< p_2 q_1 \cdots q_{j-1} < X^{\frac{11}{41}}$. Thus we obtain $p_2 q_1 \cdots q_{j-1} \in \Omega_2$ and we have a type II sum.
\end{itemize}

Similarly to the lower bound sieve case, we obtain
for any $B \geq 1$ \begin{equation}\label{UppSieveCompare}
S^{+}_{\theta}(\cA) := S(\cA,(2X)^{\frac12}) + S_4 +S_7 = S_{\theta}(\cB, (2X)^{1/2}) + T_4 + T_7 + O(X(\log X)^{-B}).    
\end{equation}

Now we define 
\begin{equation}\label{DefRho+}
\rho^{+}(n) := \rho(n) + \rho_4(n) + \rho_7(n),    
\end{equation}
where $\rho_i(n)$ is corresponding to $T_i$ and $S_i$, for example,
\begin{equation}\label{ExampleRho4}
\rho_4(n) = \sum_{\substack{n = pm \\ X^{\frac{11}{41}} \leq p <(2X)^{\frac{1}{2}} } }\rho(m,p).   
\end{equation}
Now we have finished constructing our upper bound sieve $\rho^{+}(n)$ satisfying a condition  similar to the weak Balog-Friedlander condition (\ref{weakbfcondition}). Namely (\ref{UppSieveCompare}) can be written in the following form
\begin{align}
& \notag \sum_{\substack{n \sim X \\ n \in \N^c} }cn^{1-\frac{1}{c}}\rho^{+}(n) e(n \theta) = S^{+}_{\theta}(\cA) = S^{+}_{\theta}(\cB) +O(X(\log X)^{-B})\\
= & \label{Compare} \sum_{n \sim X }\rho^{+}(n) e(n \theta) + O(X(\log X)^{-B}).
\end{align}

If $\theta = 0$, similarly to the lower bound sieve case, we can obtain the following asymptotic formula
$$
S^{+}(\cA) : = S(\cA,(2X)^{\frac12})+S_4 + S_7 = (1 + I_4 + I_7)S(\cB,(2x)^{\frac12}) + O(X(\log X)^{-B}),
$$
where
$I_4 < 1.0914$, $I_7 < 0.0045 $.
Let $\alpha^{+} = 1 + I_4 + I_7$. Hence
$$
S^{+}(\cA) = \alpha^{+}S(\cB, (2X)^{\frac12}) + O(X(\log X)^{-B}) < 2.0959 \cdot S(\cB, (2X)^{\frac12}).
$$
Now $\alpha^{+} \leq 2.0959 < 3 \alpha^{-}$, and Theorem \ref{LUthm} (iv) follows.

Now we can deduce arithmetic progressions cases of (\ref{Compare}).

\begin{lemma}\label{ApUpp}
For any $1 \leq q \leq \log^{C} X$ with $C \geq 2$ and $b \in [q]$, we have for any $B>0$,
\begin{equation}\label{UppApExp}
\sum_{\substack{n \in (X,2X] \cap \N^c \\ n \equiv b \pmod{q}}}cn^{1-\frac{1}{c}}\rho^{+}(n)e(n \theta) = \sum_{\substack{n \in (X,2X] \\ n \equiv b \pmod{q}}}\rho^{+}(n)e(n \theta) + O\left(\frac{X}{\log^B X}\right).   
\end{equation}
Furthermore, letting $\theta = 0$ and $(b,q)=1$, 
\begin{align}
\sum_{\substack{n \in (X,2X] \cap \N^c \\ n \equiv b \pmod{q}}}cn^{1-\frac{1}{c}}\rho^{+}(n) & = \sum_{\substack{n \in (X,2X] \\ n \equiv b \pmod{q}}}\rho^{+}(n) + O\left(\frac{X}{\log^B X}\right) \label{UppApNoExp}\\
& = \frac{\alpha^{+}}{\phi(q)}\int_{X}^{2X} \frac{1}{\log t} dt + O\left(\frac{X}{\log^B X}\right).\label{SiegelWalAp}
\end{align}
\end{lemma}
\begin{proof}
We have, by orthogonality of additive characters and (\ref{Compare}), 
\begin{align}
& \sum_{\substack{n \in (X,2X] \cap \N^c \\ n \equiv b \pmod{q}}}cn^{1-\frac{1}{c}}\rho^{+}(n)e(n \theta)  = \sum_{n \in (X,2X] \cap \N^c }cn^{1-\frac{1}{c}}\rho^{+}(n)e( \theta n) \cdot \frac{1}{q} \sum_{k=1}^{q}e\left(k \frac{n-b}{q}\right)\notag\\
 = & \frac{1}{q}\sum_{k=1}^{q}e\left(\frac{-bk}{q} \right)\sum_{n \in (X,2X] \cap \N^c }cn^{1-\frac{1}{c}}\rho^{+}(n)e\left( n\left(\theta + \frac{k}{q} \right) \right)\notag\\
 = &  \sum_{\substack{n \in (X,2X]  \\ n \equiv b \pmod{q}}}\rho^{+}(n)e(n \theta) + O\left(\frac{X}{\log^B X} \right).\label{SumRhoExp}
\end{align}
Hence (\ref{UppApExp}) and (\ref{UppApNoExp}) follow, and (\ref{SiegelWalAp}) follows by the Siegel-Walfisz Theorem.
\end{proof}

\begin{Rem}\label{rem4}
   If one studies Fourier transform for lower bound sieve as Matom\"aki and Shao \cite{MS17}, one can only get an asymptotic formula similar to (\ref{Compare}) for $1 < c < 8/7$ since Kumchev's argument \cite{K97} for type I and II sums required that $1<c<8/7$ and also Balog-Friedlander's argument \cite{BF92} gave a rather narrow Type I estimate, so we cannot do Buchstab's decomposition twice.
\end{Rem}

\section{Upper bound Fourier transform}\label{UBFourier}
\setcounter{lemma}{0} \setcounter{theorem}{0}
\setcounter{equation}{0}
In this section, we will verify Theorem \ref{LUthm} (\ref{UppForRestrction}) and (\ref{PsedoRandom}). The arguments in this section follow standard exponential sum estimates, which can be found, e.g., in \cite[Chapters 1-3]{Vau97}. For completeness, we provide the full details of the proofs.

Suppose that $\theta \in \T$. We first note that $1 \leq W \ll \log X$ and $(l,W)=1$ in Theorem \ref{LUthm}. By Lemma \ref{ApUpp},
$$
\sum_{\substack{n \in (X,2X] \cap \N^c \\ n \equiv l \pmod{W}}}cn^{1-\frac{1}{c}}\rho^{+}(n)e(n \theta) \quad \text{and} \quad 
\sum_{\substack{n \in (X,2X] \\ n \equiv l \pmod{W}}}\rho^{+}(n)e(n \theta)
$$
equal apart from an acceptable error. Thus next we just consider the approximation of $\sum_{\substack{n \in (X,2X]  \\ n \equiv l \pmod{W}}}\rho^{+}(n)e(n \theta)$.
We divide the torus $\T$ into major arcs and minor arcs. For the minor arc case, we use the standard type I and II estimates. For the major arc case we use Lemma \ref{ApUpp}.

Let $P$ and $Q$ be as in Theorem \ref{primecWnbT}. Recall that for $(a,q)=1$,
$$
\M_{a,q}=\left[\frac{a}{q}-\frac{1}{qQ}, \frac{a}{q}+\frac{1}{qQ}\right],
$$
and the major arcs and minor arcs are
$$
\M = \bigcup_{1 \leq q \leq P}\bigcup_{\substack{a=0 \\(a,q)=1}}^{q-1} \M_{a,q}, \quad \text{ and } \quad \m = \T \setminus \M.
$$
We will prove Theorem \ref{primecWnbT} (ii) and Theorem \ref{LUthm} (ii) for $\theta \in \m$ in Section \ref{minor} and Theorem \ref{primecWnbT} (i) and Theorem \ref{LUthm} (ii) for $\theta \in \M$ in Section \ref{major}. Recalling that Theorem \ref{LUthm} (iii) and (iv) were established in Section \ref{LBsieve} and \ref{UBsieve}, this finishes the proof of Theorem \ref{LUthm} and thus the proof of Theorem \ref{main}.
\subsection{Minor arc case}\label{minor} In this subsection, we will prove Theorem \ref{primecWnbT} (ii) and Theorem \ref{LUthm} (ii) for $\theta \in \m$. Namely, we will study, for $\theta \in \m$, the exponential sum
$$
\sum_{\substack{n \in (X,2X]  \\ n \equiv l \pmod{W}}}\rho^{+}(n)e(n \theta).
$$
In order to estimate this exponential sum, we need the following type I and II information. One can use the standard methods to get these two lemmas, see e.g. \cite[Chapter 8]{IK}.
\begin{lemma}\label{VauTypeI}
Let $X,N,W$ and $A$ be as in Theorem \ref{primecWnbT}. Suppose that $(c,q)=1$, $M \leq X^{1 - \epsilon}$, for any $\epsilon>0$, $MK \asymp X$ and $a_m \ll d^{O(1)}(m)$. Then for all $\theta \in \m$, we have
$$
 \sum_{\substack{m \sim M \\ k \sim K \\mk \equiv c \pmod{q}}}a_me(mk \theta)  \ll \frac{X}{\log^{2A} X} \ll \frac{N}{\log^A N}.
$$
\end{lemma}

\begin{lemma}\label{VauTypeII}Let $X,N,W$ and $A$ be as in Theorem \ref{primecWnbT}. Suppose that $(c,q)=1$ and $X^{\epsilon} \leq M \leq X^{1- \epsilon}$, for some $\epsilon>0$, $MK \asymp X$ and $a_m, b_m \ll d^{O(1)}(m)$. Then for all $\theta \in \m$, we have
$$
\sum_{\substack{m \sim M \\ k \sim K \\mk \equiv c \pmod{q}}}a_mb_ke(mk \theta)\ll \frac{X}{\log^{2A} X} \ll \frac{N}{\log^A N}.
$$
\end{lemma}
\begin{proof}[{\bf Proof of Theorem \ref{primecWnbT} (ii)}] Recalling (\ref{DefRho+}), we have
\begin{align*}
& \sum_{\substack{n \in (X,2X]  \\ n \equiv l \pmod{W}}}\rho^{+}(n)e(n \theta)\\
= & \sum_{\substack{n \in (X,2X]  \\ n \equiv l \pmod{W}}}\rho(n)e(n \theta) +\sum_{\substack{n \in (X,2X]  \\ n \equiv l \pmod{W}}} \rho_4(n)e(n \theta) +\sum_{\substack{n \in (X,2X]  \\ n \equiv l \pmod{W}}} \rho_7(n)e(n \theta)\\
= & S_{\rho} + S_{\rho_4} + S_{\rho_7}.
\end{align*}
We can use Vaughan's identity, see \cite[Chapter 13]{IK} to decompose $S_{\rho}$ into type I and II sums and then use Lemmas \ref{VauTypeI} and \ref{VauTypeII}. We can directly apply Lemma \ref{VauTypeII} for $S_{\rho_4}$ and $S_{\rho_7}$, since they can be decomposed into type II sums, e.g. see (\ref{ExampleRho4}).
\end{proof}

\begin{proof}[{\bf Proof of Theorem \ref{LUthm} (ii) for $\theta \in \m$}] By Lemma \ref{VauTypeI}, we have 
$$\sum_{\substack{n \sim X \\ n \equiv l \pmod{W}}}e(n \theta) \ll \frac{X}{\log^2 X}.
$$
Combining the above with the fact that Theorem \ref{primecWnbT} (ii) holds, our claim follows.
\end{proof}

\subsection{major arc case}\label{major}
In this subsection, we will prove Theorem \ref{primecWnbT} (i) and Theorem \ref{LUthm} (ii) for $\theta \in \M$ .

\begin{proof}[\bf{Proof of Theorem \ref{primecWnbT} (i)}]
Let $Y = \eta X$ with $\eta \in [{1}/{\log^A X},1]$. Recall that $\rho^{+}(n) = 0$ if $n$ has factors less than $X^{\frac{5}{41}}$. Let $1 \leq q \leq P$, $(a,q)=1$ and $(l,W)=1$. By a short interval version of  (\ref{SiegelWalAp}) (the interested reader can proved a short interval version of (\ref{SiegelWalAp}) by following the same arguments in Section \ref{UBsieve}) with $B=100A$,
\begin{align}
& \sum_{X<Wn+l\leq X+Y} \rho^{+}(Wn+l)e\left(\frac{a}{q}n\right)=\sum_{\substack{r=1 \\(Wr+l, q)=1}}^{q} e\left(\frac{a}{q}r\right)\sum_{\substack{  X < Wn+l \leq  X + Y \\ n \equiv r \pmod{q}} }\rho^{+}(Wn+l) \notag\\
 = & \frac{\alpha^{+}}{\phi(Wq)}\left(\int_{X}^{X+Y} \frac{1}{\log t} dt \right)\sum_{\substack{r=1 \\ (Wr+l, q)=1}}^{q} e\left(\frac{a}{q}r\right) + O\left(\frac{Y}{\log^{100A} N}\right) \label{CompareShortRhoE} .
\end{align}

Let $\theta \in \M_{a,q}$ for some $(a,q)=1$ and $1 \leq q \leq P$, and write $\lambda = \theta - a/q$. Now we compare 
$$\sum_{ Wn+l \sim X } \rho^{+}(Wn+l)e(\theta n)
\quad
\text{with} 
\quad
\frac{\alpha^{+}}{\phi(Wq)} \sum_{\substack{r=1 \\ (Wr+l, q)=1}}^{q} e\left(\frac{a}{q}r\right)\sum_{n \sim X}\frac{e(\lambda n)}{\log n}.
$$
To do this, let
$$
f(n) =  \rho^{+}(n)e\left(\frac{a}{q}\frac{(n-l)}{W}\right)\1_{n \equiv l \pmod{W}} - \frac{\alpha^{+}}{\phi(Wq) \log n} \sum_{\substack{r=1 \\ (Wr+l, q)=1}}^{q} e\left(\frac{a}{q}r\right).
$$
By partial summation
$$
\sum_{n \sim X} f(n) e(\lambda n) =  e(2 \lambda X)\sum_{n \sim X}f(n) - 2 \pi i\int_{X}^{2X} \lambda e(\lambda t) \sum_{X \leq n \leq t} f(n)dt.
$$
Since $$\lambda \leq \frac{1}{qQ}\leq \frac{\log^{50A} N}{N},$$ we have by (\ref{CompareShortRhoE})
$$
\sum_{n \sim N} f(n) e(\lambda n)  \ll \frac{X}{\log^{50A} X} \ll \frac{N}{\log^{49A} N}.
$$
Thus
\begin{equation}\label{SumRho+Wn+lE}
\sum_{ n \sim N} \rho^{+}(Wn+l)e(\theta n) = \frac{\alpha^{+}}{\phi(Wq)} \sum_{\substack{r=1 \\ (Wr+l, q)=1}}^{q} e\left(\frac{a}{q}r\right)\sum_{n \sim X}\frac{e(\lambda n)}{\log n} + O\left(\frac{N}{\log^{49A} N}\right).
\end{equation}
Let $G(t) = \sum_{n \leq t} e(\lambda n)$. By partial summation 
$$
\sum_{n \sim X}\frac{e(\lambda n)}{\log n}
= \frac{1}{\log 2X}G(2X) + \int_{X}^{2X}\frac{G(t)}{t \log^2 t}dt
\ll \frac{1}{\log X} \min\{X, \lambda^{-1}\}.
$$
On the other hand, since $(l,W)=1$, we have
$$
\sum_{\substack{r=1 \\ (Wr+l, q)=1}}^{q} e\left(\frac{a}{q}r\right)  = \sum_{d \mid q  }\mu(d)\sum_{\substack{r=1 \\ d \mid Wr+l}}^{q}e\left(\frac{a}{q}r\right)= \sum_{\substack{d \mid q \\ (d,W)=1}}\mu(d)\sum_{\substack{r=1 \\ r \equiv -l\overline{W_d} \pmod{d}}}^{q}e\left(\frac{a}{q}r\right),
$$
where $\overline{W_d}W \equiv 1 \pmod{d}$. The last sum on the right hand side vanishes unless $d=q$. The case $d=q$ is possible only when $(q,W)=1$. Hence
\begin{equation}\label{SumrqE}
\sum_{\substack{r=1 \\ (Wr+l,q)=1 }}^{q}e\left(\frac{a}{q}r\right) =  \begin{cases}
\mu(q)e\left(-\frac{l\overline{W}a}{q}\right) & \text{if } (W,q)=1,\\
0 & \text{otherwise}.
\end{cases}    
\end{equation}
From (\ref{SumRhoExp}), (\ref{SumRho+Wn+lE}) and (\ref{SumrqE}), we have for $\theta \in \M_{a,q}$,
\begin{equation}\label{DyadicNuWXb}
\sum_{n \sim N}\nu_{W,X,b}(n)e(n\theta) \ll \frac{N}{\phi(q)(1 + N\|\theta - a/q\|)} + O\left(\frac{N}{\log^A N}\right).    
\end{equation}
\end{proof}
\begin{Rem}
One may find the dyadic sum (\ref{DyadicNuWXb}) is different from $\hat{\nu}(\theta)$ and when sum over all $N \in [2^j,2^{j+1})$ for $j \ll \log N$ there might be one more $\log N$ factor on the right hand side in (\ref{DyadicNuWXb}). However, if we estimate the sum in (\ref{SumRho+Wn+lE}) over $N/\log^A N \leq n \leq N$ directly instead of using dyadic argument and for $1 \leq n \leq N/\log^A N$ we use the trivial upper bound, then we have
$$
\widehat{\nu}_{W,X,b}(\theta) \ll \frac{N}{\phi(q)(1 + N\|\theta - a/q\|)} + O\left(\frac{N}{\log^A N}\right).
$$
\end{Rem}
Next we will verify Theorem \ref{LUthm} (ii) for $\theta \in \M$.
\begin{proof}[\bf{Proof of Theorem \ref{LUthm} (ii) for $\theta \in \M$}]
Since $\rho^{+}(n) = 0$ if $n$ has factors less than $X^{\frac{5}{41}}$, we have $(n, qW)=1$. Writing $\theta = {a}/{q} + \lambda$, thus
\begin{equation}\label{RhoEnTheta}
\sum_{\substack{n \sim X\\ n \equiv l \pmod{W}}} \rho^{+}(n)e(n \theta)
= \sum_{\substack{r \pmod{[W,q]} \\ r \equiv l \pmod{W} \\ (r,Wq)=1} }e\left(\frac{ar}{q}\right) \sum_{\substack{n \sim X \\ n \equiv r \pmod{[W,q]}}}\rho^{+}(n)e(\lambda n).   
\end{equation}

Next we follow Green's argument \cite{G05} by dividing $\{n \in (X,2X]: n \equiv r \pmod{[W,q]}\}$ 
into $T$ arithmetic progressions $P_i$ for $1 \leq i \leq T$ such that these $P_i$ have common difference $[W,q]$ and about equal length. $T$ will be chosen later. 

Let $x_i$ be any element in $P_i$. By (\ref{SiegelWalAp})
$$
\sum_{n \in P_{i}}\rho^{+}(n) = \frac{\alpha^{+}}{\log X}\frac{[W,q]}{\phi([W,q])}|P_i| + O\left(\frac{[W,q]|P_i|}{\phi([W,q])\log^2 X}\right).
$$
Hence, the inner sum on the right hand side of (\ref{RhoEnTheta}) equals
\begin{align}
& \notag \sum_{\substack{n \sim X \\ n \equiv r \pmod{[W,q]}}}\rho^{+}(n)e(\lambda n) = \sum_{i = 1}^{T} \sum_{n \in P_i} \rho^{+}(n)e(\lambda n) \\
= & \label{CompareExp} \sum_{i = 1}^{T} e(\lambda x_i)\sum_{n \in P_i} \rho^{+}(n) + \sum_{i = 1}^{T} \sum_{n \in P_i} \rho^{+}(n)\left(e(\lambda n)-e(\lambda x_i) \right)\\
= & \notag \frac{\alpha^{+}}{\log X}\frac{[W,q]}{\phi([W,q])}\sum_{i=1}^{T}e(\lambda x_i)|P_i| + O\left(\frac{1}{\phi([W,q])}\frac{X}{\log^2 X}  \right) + O\left(\frac{1}{\phi([W,q])}\frac{X}{\log X}\frac{\lambda X}{T} \right).
\end{align}
Note that $|\lambda| \leq 1/(qQ) \leq {\log^{50A} N}/{N}$. In order to make the last error term small, we choose $T > \log^{50A+1} N$. Similarly,
$$
e(\lambda x_i)|P_i| = \sum_{n \in P_i}e(\lambda n) + \sum_{n \in P_i}(e(\lambda x_i) - e(\lambda n)) = \sum_{n \in P_i}e(\lambda n) + O\left(\frac{\lambda X|P_i|}{T}\right).
$$
Thus, we have
$$
\sum_{\substack{n \sim X \\ n \equiv r \pmod{[W,q]}}}\rho^{+}(n)e(\lambda n)=  \frac{\alpha^{+}}{\log X}\frac{[W,q]}{\phi([W,q])}\sum_{\substack{n \sim X \\ n \equiv r \pmod{[W,q]}}}e(\lambda n) + O\left(\frac{1}{\phi([W,q])}\frac{X}{\log^2 X}  \right).
$$
Thus
\begin{equation}\label{WqArgument}
\sum_{\substack{n \sim X \\ n \equiv l \pmod{W}}} \rho^{+}(n)e(n \theta) = \frac{\alpha^{+}}{\log X}\frac{[W,q]}{\phi([W,q])}\sum_{\substack{n \sim X \\ n \equiv l \pmod{W} \\  (n,Wq)=1}}e(n \theta) + O\left(\frac{\phi(q)}{\phi([W,q])}\frac{X}{\log^2 X} \right).    
\end{equation}
Comparing (\ref{WqArgument}) to Theorem \ref{LUthm} (ii), we just need to show that 
\begin{enumerate}[(a)]
    \item 
    $$
    \frac{[W,q]}{\phi([W,q])} = (1+o(1))\frac{W}{\phi(W)};
    $$
    \item 
    $$
    \sum_{\substack{n \sim X \\ n \equiv l \pmod{W} \\ (n,q) \neq 1}}1 = o\left (\frac{X}{W} \right).
    $$ 
\end{enumerate}
To prove (a) and (b), we need to consider the number of prime factors of $q$.

Recall that $W = \prod_{p \leq w}p$, with $w = 0.1\log\log X$ and $1 \leq q \leq P =\log^{10A} N$. Let $k$ denote the number of prime factors of $q$ that are at least $w$ and let $p_i$ denote the $i$-th prime number. Then
$$
k \log w \leq \sum_{i=\pi( w)+1}^{k+\pi(w)}\log p_{i} \leq \sum_{\substack{p|q \\ p>w}} \log p \leq \log q.
$$
Hence, 
$$
k \leq \frac{\log q}{\log w} \ll_A \frac{\log\log X}{\log\log\log X}
$$
and
\begin{equation}\label{Mertens}
\sum_{\substack{p \mid q \\ p > w}} \frac{1}{p} \leq \sum_{\pi(w)<i \leq k+\pi(w)}\frac{1}{p_i} \leq \frac{k}{w} \ll _{A} \frac{1}{\log\log\log X}.
\end{equation}
By (\ref{Mertens}), (a) follows since 
\begin{align*}
&\frac{[W,q]}{\phi([W,q])} = \frac{W}{\phi(W)}\prod_{\substack{p \mid q \\ p > w}} \left(1 -\frac{1}{p} \right)^{-1}\\ =&\frac{W}{\phi(W)} + O\left(\exp\left({\sum_{\substack{p \mid q \\ p > w}}\frac{1}{p}}\right) -1\right)\frac{W}{\phi(W)} = (1+o(1))\frac{W}{\phi(W)}.
\end{align*}
By the Chinese reminder theorem and (\ref{Mertens}), (b) follows since
$$
\sum_{\substack{n \sim X \\ n \equiv l \pmod{W} \\ (n,q) \neq 1}} 1 \leq \sum_{\substack{p \mid q \\ p>w}} \sum_{\substack{n \sim X \\ n \equiv l \pmod{W} \\ n \equiv 0 \pmod{p}}}1 \ll  \frac{X}{W}\sum_{\substack{p \mid q \\ p>w} }\frac{1}{p} = o\left( \frac{X}{W}\right).
$$
\end{proof}

\section{Roth's theorem for Piateski-Shapiro primes}\label{Roth}
\setcounter{lemma}{0} \setcounter{theorem}{0}
\setcounter{equation}{0}
In this section, we will prove Theorem \ref{ApRWrang} by employing another version of the transference principle given by Green-Tao \cite[Proposition 5.1]{GT06}.
\begin{proposition}[A transference principle for 3-term arithmetic progressions]\label{GT3termTran}
Let $N$ be a large prime. Suppose that $f: \Z_{N} \to \R_{+}$ satisfies the following conditions:
\begin{enumerate}[(i)]
    \item({\bf mean condition}) $\mathbb{E}_{n \in \Z_{N}}f(n) \geq \delta$, for some $\delta \in (0,1]$;
    \item({\bf pseudorandomness}) There exists a majorant $\nu:\Z_{N} \to \R_{+}$ with $f \leq \nu$ pointwise, such that $\|\widehat{\nu} - \widehat{\1}_{[N]}\|_{\infty} = o(N)$;
    \item ({\bf restriction estimate}) We have 
    $$
    \sum_{a \in \Z_{N}}\left|\sum_{n \in \Z_{N}}f(n)e\left(n \frac{a}{N} \right) \right|^q \leq K N^{q-1},
    $$
    for some fixed $q,K$ with $K \geq 1$ and $2<q<3$. 
\end{enumerate}
Then we have
$$
\E_{n,d \in \Z_{N}}f(n)f(n+d)f(n+2d) \geq c(\delta)
$$
where $c(\delta)>0$ is a constant depending only on $\delta$.
\end{proposition}
Although Proposition \ref{GT3termTran} concerns $\Z_{N}$, we will choose $f$ so that the domain of $f$ is a subset of $(N/4,N/2)$, and then the result of Proposition \ref{GT3termTran} is no different from the integer case.

Before proving Theorem \ref{ApRWrang}, we need to construct $\rho^{-}$ and $\rho^{+}$ as in Sections \ref{LBsieve} and \ref{UBsieve} to prove a lower bound result for the sum of $\rho(n)$ over $n \leq N$ in arithmetic progressions and pseudorandomness for $\rho^{+}(n)$.

Let $\gamma = 1/c \in ({205}/{243},1)$. Using the arguments of Rivat and Wu \cite{RW01}, one can prove an arithmetic progressions version of \cite[(1.3)]{RW01}. Specifically, for any $\eta \in (0,1]$, $d \leq \log X$, $(l,d)=1$ and arithmetic progression $P \subset \{n \in [X]: n \equiv l \pmod{d}\}$ with length at least $\eta \frac{X}{d}$, we have
\begin{equation}\label{3termMean}
\sum_{n \in P \cap \N^c} \rho(n) \geq \alpha^{-} \frac{d}{\phi(d)}\frac{|P|}{cX^{1-\frac1c}\log X}, 
\end{equation}
where $\alpha^{-} > 0$ .

We choose $\rho^{+} = \rho(n,z)$ as an upper bound sieve with $z = X^{6 \gamma-5}= X^{\frac{5}{81}}$, and the corresponding coefficient $$0<\alpha^{+} = \frac{81}{5} \omega\left(\frac{81}{5}\right)<10.$$ 
\begin{Rem}
The choice of $\rho^{+}$ is quite flexible because the mean condition of the main condition of Proposition \ref{GT3termTran} only requires a positive density, meaning that $\alpha^{-}/\alpha^{+}>0$ is sufficient.  
\end{Rem}
Now we can prove Theorem \ref{ApRWrang}.
\begin{proof}[Proof of Theorem \ref{ApRWrang}]
Let $\delta > 0$ be the relative upper density of $B \subset \N^c \cap \cP$,  $X>0$ be sufficiently large such that
\begin{equation}\label{limsupB}
\frac{B \cap [3X/10, 2X/5]}{|\N^c \cap \cP \cap [X]|} \geq \frac{1}{11} \delta. 
\end{equation}
Let $w = 0.01\log \log X$ and $W = \prod_{p\leq w}p$. By the pigeonhole principle, we can choose one $b \in [W]$ with $(b,W)=1$ such that 
\begin{equation}\label{densityBap}
|B \cap \{Wn+b: X/4 \leq Wn+b \leq X/2\}| \geq \frac{1}{2}\frac{1}{\phi(W)}\left|B \cap [3X/10, 2X/5]\right|.
\end{equation}

In the proof of Theorem \ref{main}, we first construct $f$ and $\nu$ similar to (\ref{Deff}) and (\ref{Defnu}). Let $N$ be a large prime such that $WN+b \in (0.99X,1.01X)$.
 For $c \in (1, {243}/{205})$, we define $f,\nu:[N] \to \R_{+}$ through
$$
f(n)=
\begin{cases}
    c \cdot\frac{\log X}{\alpha^{+}} \cdot \frac{\phi(W)}{W}(Wn + b)^{1- \frac{1}{c}}\rho(Wn+b), & \text{ if } Wn+b \in B \cap [3X/10, 2X/5],  \\
    0, & \text{ otherwise}.
\end{cases}
$$
and
$$
\nu(n)=
\begin{cases}
c\cdot\frac{\log X}{\alpha^{+}} \cdot \frac{\phi(W)}{W}(Wn + b)^{1- \frac{1}{c}}\rho^{+}(Wn+b), & \text{ if } Wn+b \in \N^{c},  \\
    0, & \text{ otherwise}.
\end{cases}
$$
By (\ref{3termMean}) with $\eta=1$, $d = W$, (\ref{densityBap}) and (\ref{limsupB}), we have 
\begin{align*}
\sum_{n \in [N]} f(n) & \gg_{c} X^{1-\frac1c}\log X \frac{\phi(W)}{W} \sum_{Wn+b \in B \cap [3X/10, 2X/5]} \rho(Wn+b)\\
& \gg_{c}  \frac{X^{1-\frac1c}\log X}{W} |B \cap [3X/10, 2X/5]|\\
& \gg_{c,\delta} \frac{X^{1-\frac1c}\log X}{W} |\N^c \cap \cP \cap [X]|.
\end{align*}
By (\ref{3termMean}) again with $\eta=1$, $d = 1$, we obtain
$$
\sum_{n \in [N]} f(n) \gg_{c, \delta} \frac{X}{W} \gg_{c, \delta} N,
$$
which implies Proposition \ref{GT3termTran} (i). Proposition \ref{GT3termTran} (ii) is a weak version of Theorem \ref{LUthm} (ii), so one can prove it as in Section \ref{UBFourier}. Proposition \ref{GT3termTran} (iii) looks slightly different from Proposition \ref{Tran} (iii). However, one can replace $[0,1)$ with $1/N$-separated points $\{0,\frac{1}{N}, \dots, \frac{N-1}{N}\}$ and then follow our arguments for proving Theorem \ref{primecWnbT} and Theorem \ref{LUthm} (i) to prove a $\mathbb{Z}_{N}$ analogy of restriction estimates. For some $K \geq 1$ and fixed $2 < q < 3$ depending on $c$, we have
$$
\sup_{\theta \in \T} \sum_{a \in \Z_{N}}\left|\sum_{n \in \Z_{N}}f(n)e\left(n \frac{a}{N} + n \theta \right) \right|^q \leq K N^{q-1}.
$$
Now we apply Proposition \ref{GT3termTran} getting that 
$$
\sum_{n,d \in \Z_{N}}f(n)f(n+d)f(n+2d) \gg N^2.
$$
Note that $f$ is supported on $(N/4,N/2)$, so $n,n+d,n+2d \in (N/4,N/2)$. By $n,n+2d \in (N/4,N/2)$ we have that $n \in (N/4, N/2)$ and $d \in [0,N/8) \cup (3N/8,5N/8)$. If $d \in (3N/8,5N/8)$, then $n+d \in (5N/8, N] \cup [0, N/8)$ which is a contradiction. Hence,
$$
N^2 \ll \sum_{n, d \in \Z_{N}}f(n)f(n+d)f(n+2d) =  \sum_{\substack{n \in (N/4,N/2) \\ d \in [0, N/8)}}f(n)f(n+d)f(n+2d).
$$
Consequently, there is a non-trivial three term arithmetic progression $Wn+b, W(n+d)+b, W(n+2d)+b \in B$, and our claim follows.
\end{proof}


\end{document}